\documentclass[10pt,amsfonts, epsfig]{amsart}
\usepackage{amsmath, amscd, amssymb}
\usepackage{graphpap, color}
\usepackage{mathrsfs}
\usepackage{pstricks}
\usepackage{color}
\usepackage{cancel}
\usepackage[mathscr]{eucal}

\usepackage{pstricks}
\usepackage{color}
\usepackage{cancel}
\usepackage{verbatim}

\def\barM{{\bcM}}

\def\barS{{\overline{S}}}

\def\tfd{{\ti\fD}}

\def\defeq{:=}\def\bB{{\mathbf B}}

\def\bC{{\mathbf C}}

\def\bcM{\overline{M}}

\numberwithin{equation}{section}
\def\sP{{\mathscr P}}

\def\oWc{{\overline W\dcirc}}

\def\sO{{\mathscr O}}

\def\sL{{\mathscr L}}

\def\sO{\mathscr{O}}

\def\sR{\mathscr{R}}

\newcommand{\A }{\mathbb{A}}
\newcommand{\CC}{\mathbb{C}}
\newcommand{\EE}{\mathbb{E}}

\newcommand{\PP}{\mathbb{P}}
\newcommand{\QQ}{\mathbb{Q}}

\newcommand{\ZZ}{\mathbb{Z}}

\def\sK{{\mathscr K}}

\def\fa{{\mathfrak a}}

\def\sP{{\mathscr P}}

\def\fS{\mathfrak S}

\def\cpx{{\cP_\cX}}

\def\ured{^{\text{red}}}

\newcommand{\bw}{\mathbf{w}}

\def\bH{{\mathbf H}}

\newcommand{\bP}{\mathbf{P}}

\newcommand{\cal}{\mathcal}

\def\cB{{\cal B}}
\def\cC{{\cal C}}
\def\cD{{\cal D}}
\def\cE{{\cal E}}
\def\cF{{\cal F}}
\def\cH{{\cal H}}

\def\cM{{\cal M}}
\def\cN{{\cal N}}
\def\cO{{\cal O}}

\def\cR{{\cal R}}

\def\cU{{\cal U}}

\def\cX{{\cal X}}
\def\cY{{\cal Y}}

\def\fm{\mathfrak{m}}
\def\fp{\mathfrak{p}}

\def\mapright#1{\,\smash{\mathop{\lra}\limits^{#1}}\,}

\def\dual{^{\vee}}

\def\sta{^\ast}

\def\virt{^{\mathrm{vir}}}
\def\upmo{^{-1}}
\def\sta{^{\ast}}
\def\dpri{^{\prime\prime}}

\def\pri{^{\prime}}

\def\sta{^*}

\def\lra{\longrightarrow}
\def\lpri{_{\mathrm{pri}}}
\def\lsta{_{\ast}}

\newcommand{\si}{\sigma}

\newcommand{\lmu}{_\mu}

\def\begeq{\begin{equation}}
\def\endeq{\end{equation}}
\def\and{\quad{\rm and}\quad}
\def\bl{\bigl(}
\def\br{\bigr)}
\def\defeq{:=}

\def\sub{\subset}
\def\Ao{{\mathbb A}^{\!1}}

\def\Po{{\mathbb P^1}}
\def\and{\quad\text{and}\quad}

\def\lalp{_\alpha}

\def\ob{\text{ob}}

\DeclareMathOperator{\pr}{pr}

\DeclareMathOperator{\image}{Im} 
\DeclareMathOperator{\id}{id}

 \DeclareMathOperator{\rank}{rank}
\DeclareMathOperator{\spec}{Spec}

\newtheorem{prop}{Proposition}[section]
\newtheorem{theo}[prop]{Theorem}
\newtheorem{lemm}[prop]{Lemma}
\newtheorem{coro}[prop]{Corollary}

\newtheorem{defi}[prop]{Definition}

\newtheorem{defi-prop}[prop]{Definition-Proposition}

\def\Ob{\cO b}

\def\lloc{_{\mathrm{loc}}}

\def\loc{{\mathrm{loc}}}

\def\bul{^\bullet}

\def\ev{\text{ev}}

\def\tcX{\tilde{\cX}}
\def\tcY{\tilde{\cY}}
\def\tcD{\tilde{\cD}}

\def\sta{^\ast}
\def\image{\text{Im}\,}

\let\lab=\label

\def\sO{{\mathscr O}}

\def\sH{{\mathscr H}}
\def\sR{{\mathscr R}}

\def\beq{\begin{equation}}
\def\eeq{\end{equation}}

\def\vsp{\vskip5pt}
\def\Pf{{\PP^4}}

\def\bD{{\mathbf D}}

\def\bee{\begin{equation}}
\def\eeq{\end{equation}}

\def\Tot{\mathrm{Total}}

\def\bV{{\mathbf V}}

\let\eps=\epsilon

\def\ti{\tilde}

\def\lD{_{\Delta}}
\def\lgst{_{\rm gst}}
\def\lell{_{\rm pri}}

\def\cfM{{\cM^{\mathrm w}}}
\let\cfD=\cD
\def\tfd{{\tilde \cD}}

\def\cpx{{\tilde \cY}}
\def\ticfM{\ti\cM^{\text{w}}}
\def\tcX{{\ti\cX}}
\def\tcy{{\ti\cY}}
\def\bcpri{\bC_{\mathrm{pri}}}
\def\bcgst{\bC\lgst}
\def\tcygst{{\tcy\lgst}}

\def\lab#1{\label{#1}[{#1}]\  }
\let\lab=\label

\title[Reduced genus one Gromov-Witten invariants of quintics]{An algebraic proof of the hyperplane property of the
genus one GW-invariants of quintics
}

\date{}

\author{H-L. Chang and J. Li}

\begin{document}
\maketitle


\begin{abstract} Li-Zinger's {\sl hyperplane theorem} states that the genus one GW-invariants of the quintic threefold is the sum of its reduced genus one GW-invariants and $1/12$ multiplies of its genus zero GW-invariants. We apply the Guffin-Sharpe-Witten's theory (GSW theory) to give an algebro-geometric proof
of the {\sl hyperplane theorem}, including separation of contributions and  computation of $1/12$.
\end{abstract}

\section{Introduction}
GW-invariants of a smooth projective variety $X$ are virtual enumeration of stable maps to $X$. Let $d\in H_2(X,\ZZ)$,
and let $\bcM_{g}(X,d)$ be the moduli of genus $g$ stable morphisms to $X$; it is a
proper, separated DM-stack, and carries a canonical virtual cycle $[\bcM_{g}(X,d)]\virt$.
When $X$ is a Calabi-Yau threefold, this class is a dimension zero class, and the degree $d$ genus $g$ GW-invariants of $X$ are
$$N_g(d)_X=\deg [\bcM_{g}(X,d)]\virt.
$$

Investigating GW-invariants of Calabi-Yau threefolds is one of the main focus in the subject of Mirror Symmetry.
The subclass of genus zero invariants is largely known by now, thanks to the works of \cite{Ko, Gi,LLY}, in
case the Calabi-Yau threefold $X$ is a complete intersection in a product of projective spaces $\PP$,
which can be realized as an equivariant integral of the top Chern class of a vector bundle $\EE_{0,d}$
on the moduli of stable morphisms to $\PP$:
\beq\lab{hyper}
[\bcM_{0}(X,d)]\virt=\int_{[\bcM_{0}(\PP,d)]} c_{\text{top}}(\EE_{0,d}).
\eeq
We call this property the ``{\sl hyperplane property of the GW-invariants}''.
The techniques developed in \cite{Gi,LLY} allow one to completely solve the genus zero invariants for such $X$.

Generalizing this to high genus requires new approach, in part because the {\sl hyperplane property} stated fails for
positive genus invariants. 
In \cite{LZ} Zinger and the second named author
introduced the reduced $g=1$ GW-invariants $N_1(d)^{\text{red}}_Q$ of the quintic $Q\sub\Pf$ by applying the hyperplane peoperty to the primary component of
$\bcM_{1}(\Pf,d)$, and proved that the reduced $g=1$ GW-invariants $N_1(d)\ured_Q$
relates to the ordinary GW-invaraints $N_1(d)_Q$ by a simple linear relation.

In this paper, we  give an algebraic proof of this theorem.

\begin{theo}\rm{(Li-Zinger)}\label{main-theorem}
The reduced and the ordinary GW-invariants of a quintic Calabi-Yau $Q\sub\Pf$ are related by
$$N_1(d)_Q=N_1(d)_Q^{\text{red}}+\frac{1}{12} N_0(d)_Q.
$$
\end{theo}

In algebraic geometry, the reduced genus one GW-invariants of quintics take the following form. Let
$f: \cC\to \Pf$ and $\pi:\cC\to \bcM_{1}(\Pf,d)$ be the universal family of $\bcM_{1}(\Pf,d)$;
let $\bcM_{1}(\Pf,d)_{\text{pri}}\sub \bcM_{1}(\Pf,d)$ be the primary component that is the closure of
all stable morphisms with smooth domains. We then pick a DM-stack $\tcX\lpri$ and a proper birational morphism
$\varphi: \tcX\lpri\to \bcM_{1}(\Pf,d)_{\text{pri}}$ so that with $\ti f\lpri: \ti\cC\lpri\to \Pf$ and $\ti\pi\lpri: \ti\cC\lpri\to \tcX\lpri$
the pull back of $(f,\pi)$, the direct image sheaf
$$\ti\pi_{\text{pri}\ast} \ti f\lpri\sta \sO_{\Pf}(5)
$$
is locally free on $\tcX\lpri$. According to \cite{VZ}, (see also \cite{HL1},) such $\varphi$ exists.
We state the working
definition of the reduced invariants of $Q$.

\begin{defi}
We define the reduced $g=1$ GW-invariants of $Q$ be
\beq\label{def-red}
N_1(d)\ured_Q=\int_{[\tcX\lpri] } c_{\text{top}}(\ti\pi_{\mathrm{pri}\ast}  \ti f\lpri\sta\sO_{\Pf}(5)).
\eeq
\end{defi}

To prove Theorem \ref{main-theorem}, one needs to separate $[\barM_g(Q,d)]\virt$ into its 
``primary" part and its ``ghost" part. One then shows that the ``primary" part can be evaluated using
\eqref{def-red}, and shows that the ``ghost" part contributes $\frac{1}{12}$ multiple of the genus zero GW-invariants.

The original proof of this theorem uses analytic method, which succeeds the desired separation
by perturbing the complex structure of $Q$ to a generic almost complex structure, through hardcore analysis \cite{Zi}.

The proof worked out in this paper uses the Guffin-Sharpet-Witten's invariants. The GSW invariants originates from Guffin and Sharpe's pioneer work in \cite{Sharpe}.
Later an algebraic geometric treatment is given and GSW invariants can be realized as the GW-invariants of stable maps with $p$-field
developed in \cite{CL}. Since the structure of the moduli of stable maps to $\Pf$ with fields is much easier to analyze, it allows us to separate the ``primary" and the ``ghost" contributions algebraically. We are confident that this method can be
pushed to higher genus GW-invariants of quintics and other complete intersections of product of projective spaces.

\vsp

We now outline our proof.
Given a positive integer $d$, we form
the moduli $\bcM_1(\Pf,d)^p$ of genus $1$ degree $d$ stable morphisms to $\Pf$
with $p$-fields:
$$\bcM_1(\Pf,d)^p=\{[u,C,p]\ \big| \ [u,C]\in \bcM_1(\Pf,d),\ p\in \Gamma(u\sta \sO_\Pf(-5)\otimes \omega_C)\,\}/\sim.
$$
It is a Deligne-Mumford stack, with a
perfect obstruction theory.  The
polynomial $\bw=x_1^5+\ldots+x_5^5$ (or any generic quintic polynomial) induces a cosection
(homomorphism) of its obstruction sheaf
\beq\label{si0}
\sigma : \Ob_{\bcM_1(\Pf,d)^p}\lra \sO_{\bcM_1(\Pf,d)^p},
\eeq
whose non-surjective locus (called the degeneracy locus) is
$$\bcM_1(Q,d)\sub \bcM_1(\Pf,d)^p, \quad Q=(x_1^5+\ldots+x_5^5=0)\sub \Pf,
$$
which is proper.  The cosection localized virtual class construction of Kiem-Li
defines a  localized virtual cycle
$$[\bcM_1(\Pf,d)^p]\virt_\sigma\in A_0 \bcM_1(Q,d).
$$
(For convention of cosection localized virtual cycles, see discussion after \eqref{sigb}.)
The GW-invariant of $\bcM_1(\Pf,d)^p$ is
$$N_1(d)^p_\Pf=\deg [\bcM_1(\Pf,d)^p]\virt_\sigma.
$$

\vsp
\begin{theo}[\cite{CL}] \label{thm1.2}
For $d>0$,
the GW-invariant of $\bcM_1(\Pf,d)^p$coincides with the GW-invariant $N_1(d)_Q$ of the quintic $Q$ up to a sign:
$$N_1(d)^p_\Pf=(-1)^{5d}\cdot N_1(d)_Q.
$$
\end{theo}

By this theorem, to prove Theorem \ref{main-theorem}
it suffices to study the separation of the cycle
$[\bcM_1(\Pf,d)^p]\virt_\sigma$. Following \cite{VZ} and \cite{HL1}, we perform a modular blowing-up $\tcy$ of $\barM_1(\Pf,d)^p$.
The blown-up $\cpx$ is a union of $\cpx\lpri$, which is birational to the primary component of $\bcM_1(\Pf,d)$,
and other smooth components $\cpx_\mu$, indexed by partitions $\mu$ of $d$:
$$ \cpx=\tcy\lgst\cup(\cup_{\mu\vdash d} \cpx_\mu).
$$
Geometrically, general elements of $\cpx\lpri$ are stable morphisms in $\barM_1(\Pf,d)$ having smooth domains; general elements of
$\cpx_\mu$ (with $\mu=(d_1,\cdots,d_\ell)$) consist of \\
(1). stable morphisms $[u,C]\in\barM_1(\Pf,d)$ whose domains $C$ are smooth elliptic curves attached on with $\ell$
rational curves, and the morphisms $u$ are constant along the elliptic curves and have degree $d_i$ along the $i$-th
rational curves; \\
(2). a $p$-field $\psi\in \Gamma(C,u\sta\sO(-5)\otimes\omega_C)$, and \\
(3). auxiliary data from the modular blowing-up.


%

Let $\cD$ be the Artin stack of pairs $(C,L)$ of degree $d$ line bundles $L$ on arithmetic genus one nodal curves $C$.
We perform a parallel modular blowing-up of $\cD$ to obtain $\ti\cD\to\cD$. The pullback of $\sO(1)$ via the universal morphism
on $\tcy$ induces a morphism $\tcy\to\ti\cD$.
By working out the relative perfect obstruction theory of $\tcy\to\ti\cD$, we obtain its obstruction complex $E_{\tcy/\ti\cD}$.
The intrinsic normal cone $\bC_{\cpx/\tfd}$ of $\cpx\to \tfd$ is embedded in $h^1/h^0(E_{\cpx/\tfd})$.

By picking a homogeneous quintic polynomial, say $x_1^5+\cdots+x_5^5$, we construct a cosection (homomorphism)
$\ti\si: h^1/h^0(E_{\tcy/\ti\cD})\to\sO_{\tcy}$.
By a cosection localized version of \cite[Thm 5.0.1]{Costello} we have
\beq\label{blowup-vir}
\deg [\barM_1(\Pf,d)^p]\virt_{\si} = \deg 0^!_{\ti\sigma,\text{loc}}[\bC_{\cpx/\tfd}] . 
\eeq

Using the explicit local defining equation of $\tcy$, which is obtained following the work of \cite{Z, HL}, we conclude that
the cone $\bC_{\cpx/\tfd}$ is separated into union
$$[\bC_{\cpx/\tfd}] =[\bC\lell]+\sum_{\mu\vdash d}\,  [\bC_\mu],
$$
where $\bC\lell$ is an irreducible cycle lies over $\cpx\lgst$; each $\bC_\mu$ lies over
$\cpx_\mu$. Thus
\beq\label{sep-vir}
\deg 0^!_{\ti\sigma,\text{loc}}[\bC_{\cpx/\tfd}]  =0^!_{\ti\sigma,\loc}[\bC\lell]+\sum_{\mu\vdash d} 0^!_{\ti\sigma,\loc}[\bC_\mu].
\eeq


In Section four, we show that $0^!_{\ti\sigma,\loc}[\bC\lpri]$ equals the reduced GW-invariant
defined in (\ref{blowup-vir}); in Section six, we show that $0^!_{\ti\sigma,\loc}[\bC_\mu]=0$ for all partitions $\mu\ne(d)$, where $(d)$
is the partition of $d$ into a single block, and finally in Section seven, we prove that $0^!_{\ti\sigma,\loc}[\bC_{(d)}]=\frac{(-1)^{5d}}{12} N_0(d)_Q$.
Together with (\ref{blowup-vir}),(\ref{sep-vir}) and Theorem \ref{thm1.2}, we prove the main Theorem \ref{main-theorem}.

\vsp

The authors' original (algebro-geometric) approach to prove Theorem \ref{main-theorem} is to introduce an auxiliary substack of
$\barM_1(\Pf,d)$ to capture the contribution $\frac{1}{12} N_1(d)_Q$. Using this, the authors showed that the union of this auxiliary
substack with the primary component of $\barM_1(\Pf,d)$ contains $\barM_1(Q,d)$, and has (in principle) a perfect relative obstruction
theory, thus it has the ``hyperplane property'' induced by the relative obstruction theory.
However, the theories developed to study the deformation theory
of the union of this auxiliary substack with the primary component of $\barM_1(\Pf,d)$ were far from satisfactory, and
were never finalized.

After the discovery of the GW-theory of stable maps with $p$-field (Guffin-Sharpe-Witten invariants) and its equivalence with the GW-invariants of quintics,
we use the new moduli spaces and their localized virtual cycles to
provide the current proof of Theorem \ref{main-theorem}. We expect that this approach can be generalized to prove
a conjecture of Zinger and the second named author on high genus reduced GW-invariants of quintics 
and other complete intersection Calabi-Yau threefolds in the product of projective spaces. 
For instance, the partial blowing-up of moduli of genus two stable maps to $\PP^n$ \cite{HL2} should
lead to a proof of this conjecture for the genus two case.

\vsp
 
{\bf Acknowledgement}.  The authors thank B. Fantechi, A. Kresch, Y-B. Ruan, R. Vakil and A. Zinger for helpful discussions and explanations of their results. The first author also thanks A. Tanzini and G. Bonelli for introduction to mirror symmetry. The first named author is partially supported by Hong Kong GRF grant 600711. The second named author is partially supported by NSF-0601002
and DARPA HR0011-08-1-0091.

\section{Moduli of stable morphisms with fields}
\def\upf{^{\oplus 5}}
\def\umtf{{-\otimes 5}}

We recall the moduli of stable morphisms with fields and its associated invariants introduced in \cite{CL}.
To simplify the notation, we will focus on the genus one case.

We begin with the moduli of stable maps.
Let $\barM_1(\Pf,d)$ be the moduli of genus one degree $d$ stable maps to $\Pf$.
In this paper, we abbreviate it to $\cX\defeq \barM_1(\Pf,d)$,
with $g=1$ and $d$ implicitly understood.
We denote by 
$$(f_{\cX},\pi_{\cX}): \cC_{\cX}\to \Pf\times\cX
$$
its universal family of $\cX$.

We recall the definition of the Moduli of stable morphisms with fields.

\begin{defi-prop}
We let $\barM_1(\Pf,d)^p$ be the groupoid that associates to any scheme $S$ the set
$\barM_1(\Pf,d)^p(S)$ of all
$S$-families $(\cC_S,f_S, \psi_S)$ where
(1). $[\cC_S,f_S]\in \barM_1(\Pf,d)(S)$,
and (2). $\psi_S\in \Gamma(\cC_S, f_S\sta \sO(-5)\otimes\omega_{\cC_S/S})$.
An arrow from $(\cC_S,f_S, \psi_S)$ to $(\cC'_S,f'_S, \psi'_S)$ consists of an arrow $\eta$ from $(\cC_S,f_S)$ to $(\cC'_S,f_S')$
such that $\eta\sta(\psi_S')=\psi_S$.

The groupoid $\barM_1(\Pf,d)^p$ is a
separated Deligne-Mumford stack of finite type.
\end{defi-prop}

In \cite{CL}, this moduli space is constructed as the direct image cone over the  the Artin stack of
pairs of curves with line bundles.
Let $\cfD$ be the groupoid whose objects $\tau\in \cD(S)$, $S$ is a scheme, are pairs $(\cC_\tau,\sL_\tau)$, where
$\pi_\tau: \cC_\tau\to S$ is a flat family of genus one connected nodal curves and $\sL_\tau$ a fiberwise degree $d$
line bundle on $\cC_\tau$;
an arrow from $\tau$ to $\tau'$ in $\cD(S)$ consists of an $S$-isomorphism $\phi_1: \cC_\tau\to\cC_{\tau'}$ and an isomorphism
$\phi_2: \sL_\tau\to \phi_1\sta\sL_{\tau'}$.
Note that the invertible sheaf
$\sL_{\cX}\defeq f_{\cX}^\ast\sO(1)$
on $\cC_{\cX}$ induces a tautological morphism $\cX\to\cfD$.

We let $\pi: \cC\to\cD$ with $\sL$ on $\cC$ be the universal curve and line bundle of
$\cD$; we introduce an auxiliary invertible sheaf $\sP=\sL^{\otimes(-5)}\otimes\omega_{\cC/\cD}$.
We define the direct image stack
\beq\label{dir-C}
C(\pi\lsta(\sL^{\oplus 5}\oplus  \sP))
\eeq
to be the category whose objects in $C(\pi\lsta (\sL^{\oplus 4}\oplus \sP))(S)$
consists $(\cC_\tau,\sL_\tau, u_i,\psi)$, where $\tau\in \cD(S)$ with $\pi_\tau: \cC_\tau\to S$ and $\sL_\tau$ on $\cC_\tau$ its
associated families, $(u_1,\cdots, u_5)\in \Gamma(\pi_{\tau\ast}\sL_\tau^{\oplus 5})$ and $\psi\in \Gamma(\pi_{\tau\ast}\sP_\tau)$.
An arrow from $(\cC_\tau,\sL_\tau, u_i,\psi)$ to $(\cC_{\tau'},\sL_{\tau'}, u_i',\psi')$ (over the same $S$)
consists of an arrow $\tau\to \tau'$ in $\cD(S)$ plus the
identity of $(u_i,\psi)$ with the pullback of $(u'_i,\psi')$ under the mentioned arrow $\tau\to \tau'$.
By construction, $C(\pi\lsta(\sL^{\oplus 5}\oplus  \sP))$ is a stack over $\cD$.

For simplicity, in this paper we abbreviate $\cY=\barM_1(\Pf,d)^p$, with $g=1$ and $d$ implicitly understood.
Like before, we let
$$(f_\cY, \pi_\cY): \cC_\cY\lra \Pf\times\cY,\quad \psi_\cY\in\Gamma(\cC_\cY, \sP_\cY)\and
\sP_\cY=f_\cY\sta\sO(-5)\otimes \omega_{\cC_\cY/\cY}
$$
be the universal family of $\cY$.
We denote $\sL_\cY=f_\cY\sta\sO(1)$. After fixing a homogeneous coordinates $[z_1,\cdots, z_5]$ of $\Pf$,
the morphism $f_\cY$ is given by $(u_{\cY, i})$, $u_{\cY,i}=f_\cY\sta z_i$.
The data $(\cC_\cY, \sL_\cY, u_{\cY,i}, \psi_\cY)$
induces an open embedding $\barM_1(\Pf,d)^p\to C(\pi\lsta(\sL^{\oplus 5}\oplus  \sP))$.
We let $p$ be the composite
$$p:\cY= \barM_1(\Pf,d)^p\sub C(\pi\lsta(\sL^{\oplus 5}\oplus  \sP))\lra \cD.
$$

After working out the obstruction theory of $C(\pi\lsta(\sL^{\oplus 5}\oplus  \sP))$ relative to $\cD$
(cf. \cite[Prop 3.1]{CL}), we obtain
a perfect relative obstruction theory (of $\cY\to \cfD$)
\beq\label{y-o}
\phi_{\cY/\cfD}:(E_{\cY/\cD})\dual 
\lra L_{\cY/\cfD}^\bullet  ,\quad E_{\cY/\cfD}\defeq   R^\bullet \pi_{\cY\ast}(\sL_\cY^{\oplus 5}\oplus \sP_\cY).
\eeq
In the same spirit, a relative perfect obstruction theory
of $\cX\to\cD$ (cf. \cite[Prop 2.5 and 2.7]{CL}) is (denoting $\sL_\cX=f_\cX\sta\sO(1)$):
\beq\label{x-o}\phi_{\cX/\cfD}: ( E_{\cX/\cD})\dual  \lra L^\bullet_{\cX/\cfD},\quad  E_{\cX/\cD}\defeq
R^\bullet\pi_{\cX\ast} \sL_{\cX}^{\oplus 5}.
\eeq

According to the convention, we call the cohomology sheaf
$$\Ob_{\cY/\cD}\defeq  H^1(E_{\cY/\cfD})=R^1 \pi_{\cY\ast}(\sL_\cY^{\oplus 5}\oplus \sP_\cY)
$$
the relative obstruction sheaf of $\phi_{\cY/\cfD}$.


\vsp

Since $\cY$ is non-proper, we use Kiem and the second named author's cosection localized virtual
class to construct its GW-invariants.
In \cite{CL}, the authors have constructed a cosection of $\Ob_{\cY/\cfD}$; namely, a homomorphism
\beq\label{cos}
\sigma: \Ob_{\cY/\cfD}\lra \sO_\cY,
\eeq
based on a choice of
a quintic polynomial, say $\bw=x_1^5+\cdots+x_5^5$. It was verified in the same paper that this cosection
lifted to a cosection $\bar\sigma: \Ob_\cY\to\sO_\cY$
of the (absolute) obstruction sheaf $\Ob_\cY$, where $\Ob_\cY$ is defined by
the exact sequence
$$p\sta\Omega_\cfD\dual\lra \Ob_{\cY/\cfD}\lra \Ob_\cY\lra 0.
$$
It was also verified that the degeneracy locus $D(\sigma)$ of $\sigma$, which is the locus where
$\sigma$ is not surjective, is the closed subset
\beq\label{emb}
D(\si)=\barM_1(Q,d)\sub \barM_1(\Pf,d)^p=\cY.
\eeq
Here $Q=(\bw=0)\sub\Pf$ is the smooth Calabi-Yau quintic defined by the
vanishing of the quintic polynomial $\bw=x_1^5+\cdots+x_5^5$ used to construct the cosection $\sigma$;
$\barM_1(Q,d)$ is the moduli of stable morphisms to $Q$, and the embedding is via
the tautological embedding $\barM_1(Q,d)\sub \barM_1(\Pf,d)$ composed with the embedding
$\barM_1(\Pf,d)\sub\cY=\barM_1(\Pf,d)^p$ defined by assigning  $\psi=0$.

The cosection $\sigma$ induces a morphism of bundle-stack (see \cite{BF} for bundle stacks)
\beq\label{sigb} \si: h^1/h^0(E_{\cY/\cfD})\lra \sO_\cY
\eeq
that is surjective over $\cU=\cY-D(\si)$. (By abuse of notation, we use the same $\si$;
also, we use $\sO_\cY$ to denote the rank one trivial line bundle on $\cY$.)  We let
$$h^1/h^0(E_{\cY/\cfD})_{\si}=\bl h^1/h^0(E_{\cY/\cfD})\times_{\cY}D(\si)\br \cup \ker\{\si|_{\cU}:
h^1/h^0(E_{\cY/\cfD})|_{\cU}\to \sO_\cU\},
$$
endowed with reduced stack structure.

Applying \cite{KL} on cosection localized 
virtual class, we know that the intrinsic normal cone (cycle)
$[\bC_{\cY/\cD}]\in Z\lsta h^1/h^0(E_{\cY/\cfD})$ lies in
\beq\label{C-co}
[\bC_{\cY/\cD}]\in Z\lsta h^1/h^0(E_{\cY/\cfD})_{\si};
\eeq
applying the localized Gysin map
$$0^!_{\sigma,\mathrm{loc}}: A\lsta h^1/h^0(E_{\cY/\cfD})_{\sigma}\lra A_{\ast-n} D(\sigma),
$$
where $-n=\rank E_{\cY/\cfD}$, we obtain a localized virtual class
\beq\label{l-vir}
[\cY]\virt\lloc\defeq 0^!_{\si,\loc}[\bC_{\cY/\cD}]\in A_0 D(\si)=A_0\barM_1(Q,d).
\eeq
We define its degree to be the GW-invariants of stable morphisms to $\Pf$ with fields:
$$N_1(d)^p_\Pf=\deg \, [\cY]\virt\lloc.
$$
This is well-defined since $\barM_1(Q,d)$ is proper.


\begin{theo}[\cite{CL}] \label{theorem} Let $N_1(d)_Q$ be the GW-invariants of genus one degree $d$
stable morphisms to $Q$, then we have
$$N_1(d)^p_{\Pf}=(-1)^{5d}N_1(d)_Q.
$$
\end{theo}

We remark that this Theorem holds for all genus $g$. For our purpose, we only state in the case $g=1$.

\vsp

We will use the modular blow-up of $\cX=\barM(\Pf,d)$ to study $N_1(d)^p_\Pf$.
The version we use is that worked out by Hu and the second named author \cite{HL1},
following the original construction of Vakil-Zinger's modular blow-up of the primary component of
$\cX$ \cite{VZ}.
\vsp

Let $\cfM$ be the Artin stack of weighted genus one curves. (A weighted
nodal curve is a nodal curve with its irreducible components assigned non-negative weights; the sum of
the weights of all irreducible components is called the total weight.)
Replacing $\sL$ (of the universal family $(\cC,\sL)$ of $\cD$) by its degrees along irreducible components
of fibers of $\cC\to\cD$, we obtain an induced morphism $\cD\to\cfM$. Let
\beq\label{X-Mw} \cX\lra \cfM,\quad \cY\lra \cfM
\eeq
be the composites of $\cX,\,\cY\to\cD$ with $\cD\to\cfM$.

Let $\ticfM$ be modular blow up of $\cfM$
described in \cite{HL1}. (Since we do not need the explicit form of this blow-up in this paper, we will be
content to describe the consequence of this blow-up.)
We define the modular blow-up of $\cX$ and $\cY$ to be
$$\ti\cX=\cX\times_\cfM {\ticfM},\quad
\ti\cY=\cY\times_\cfM {\ticfM},\and \ti\cD=\cD\times_{\cfM}\ticfM.
$$
We denote $\cC_\tcy=\cC_\cY\times_\cY \tcY$ and let $f_\tcy:\cC_\tcy\to\Pf$ be the composition
of $f_\cY$ with $\cC_\tcy\to \cC_\cY$. We call $(f_\tcy,\pi_\tcy):\cC_\tcy\to\Pf\times\tcy$
the tautological family of $\tcy$.

Let $\zeta: \ti \cY\to\cY$ be the projection. Since $\ti\cY$ is derived from $\cY$ by a base change, we endow
the relative obstruction of $\ti\cY\to\ti\cD$ the pullback of that of $\cY\to\cD$:
\beq\label{ty-o}
\phi_{\ti\cY/\ti\cD}: (E_{\ti\cY/\ti\cD})\dual\lra L\bul_{\ti\cY/\ti\cD},\quad E_{\ti\cY/\ti\cD}=\zeta\sta E_{\cY/\cD}.
\eeq

 Thus, $\Ob_{\ti\cY/\ti\cD}=\zeta\sta\Ob_{\cY/\cD}$;  the cosection $\si$ pullbacks to a cosection
\beq\label{ti-si}
\ti\si=\zeta\sta\si: \Ob_{\ti\cY/\ti\cD}\to\sO_{\ti\cY},
\eeq
and the degeneracy locus
$$D(\ti\si)=D(\si)\times_\cY\ti\cY=\barM_1(Q,d)\times_{\cfM}\ticfM,
$$
which is proper. We define $h^1/h^0(E_{\ti\cY/\ti\cD})_{\ti\si}$ parallel to that defined after \eqref{sigb}. Then by \eqref{C-co}, we have
$$[\bC_{\ti\cY/\ti\cD}]\in Z\lsta h^1/h^0(E_{\ti\cY/\ti\cD})_{\ti\si}.
$$

We define $[\cpx]\virt\lloc=0^!_{\ti\si,\loc}[\bC_{\tcy/\ti\cD}] \in A_0 D(\ti\si)$ to be the cosection localized virtual class.

\begin{prop} \label{equal}
We have identity
$$\deg [\cpx]\virt\lloc=\deg [\cY]\virt\lloc=(-1)^{5d}N_1(d)_Q.
$$
\end{prop}

\begin{proof} We offer two proofs. The first is to use Theorem \ref{theorem}
and to prove $\deg\,[\ti\cY]\virt\lloc=\deg\, [\cY]\virt\lloc$;
the later is proved using the cosection localization version of \cite[Thm 5.0.1]{Costello},
and that $\ti\cD\to\cD$ is a birational morphism between two smooth Artin stacks.

The second proof is to show that $\deg\,[\ti\cY]\virt\lloc=\deg\, [\ti\cX]\virt$, which can be proved
parallel to the proof of $\deg\,[\cY]\virt\lloc=\deg\, [\cX]\virt$ worked out in \cite{CL}, and with
$\deg\,[\ti\cX]\virt=\deg\, [\cX]\virt$ using \cite[Thm 5.0.1]{Costello}.

Since both proofs are repetition of the known proofs with routine modifications, we skip the details.
\end{proof}

\section{The decomposition of cones}

Following \cite{HL1}, we know that the blown-up $\ti \cX$ is a union of smooth Deligne-Mumford stacks: one is the proper transform of
$\cX\lell\sub\cX$, which we denote by $\ti\cX\lell$; the others are indexed by partitions $\mu$ of $d$,
denoted by $\ti\cX_\mu$.
Geometrically, a generic element of $\ti\cX\lell$ is a stable morphism whose domain is smooth; a generic element of
$\ti\cX_\mu$ of $\mu=(d_1,\cdots,d_\ell)$ is a stable morphism whose domain is $\ell$
rational curves attached to a smooth elliptic curve so that the stable morphism is constant along the elliptic curve and has degree $d_i$ along the $i$-th
attached rational curve.

The corresponding stack $\ti\cY$ has similar structure. First, we have the induced projection
$$ \fp :\ti\cY\lra \ti\cX.
$$
We let
$\ti\cY_\alpha=\ti \cY\times_{\ti\cX}\ti\cX\lalp$, where $\alpha=\text{pri}$ or $\mu\vdash d$;
we denote $\ti\cX\lgst=\cup_{\mu\vdash d}\ti\cX_\mu\sub\tcX$ and $\ti\cY\lgst= \cup_{\mu\vdash d}\ti\cY_\mu$ ($=\tcy\times_{\tcX}\tcX\lgst$).
Thus 
\beq\label{normal}
\ti\cX=\ti\cX\lell\cup \ti\cX\lgst\and
\ti\cY=\ti\cY\lell\cup\ti\cY\lgst.
\eeq

We group the property of this decomposition as follows.

\begin{prop}\label{coordinate}
For any closed $y\in \ti\cY$, we can find an \'etale neighborhood $\ti X\to \tcX$ of $x=\fp(y)\in\ti\cX$,
an embedding $\ti X$
into a smooth affine scheme $Z$, and an embedding $\ti Y:=\ti X\times_{\ti\cX} \ti\cY\to Z':=Z\times \Ao$, such that
\begin{enumerate}
\item let $p_Z: Z'=Z\times\Ao\to Z$ be the first projection; then $\ti Y\to \ti X$ commutes with $p_Z: Z'\to Z$; 

\item there are regular functions $z_1,\ldots, z_4$
and $w_\mu\in\Gamma(\sO_{Z})$,  $\mu\vdash d$, such that all subschemes $(z_i=0)$, and $(w_\mu=0)\sub Z$  are
smooth; further, let $w=\prod_\mu w_\mu$, then $(z_1\cdots z_4\cdot w=0)\sub Z$ has normal crossing singularities;

\item there is a smooth morphism $Z\to \cfM$ so that the composite $\ti Y\to Z\to \cfM$ is identical to the composite
$\ti Y\to \tcy\to \cfM$; further, locally the functions $w_\mu$ are pullbacks of functions on $\cfM$;

\item let $t\in \Gamma(\sO_{Z'})$ be the pull back of the standard coordinate function of
$\Ao$ via $p_{\Ao}: Z'\to \Ao$;
then as subschemes of $Z'$,  $\ti Y=(w\cdot z_{1\le i\le 4}=w\cdot t=0)$\footnote{We
use $w\cdot z_{1\le i\le 4}=0$ to mean $w\cdot z_1=\cdots=w\cdot z_4=0$.};
$\ti Y\lell=(z_{1\le i\le 4}=t=0)$; $\ti Y_\mu=(w_\mu=0)$, and $\ti Y\lgst=(w=0)$;

\item identify $Z\cong Z'\cap(t=0)$; then as subschemes of $Z$, $\ti X= \ti Y\cap (t=0)$; $\ti X\lell=\ti Y\lell$, and $\ti X_\mu=
\ti Y_\mu\cap (t=0)$;
\end{enumerate}
\end{prop}

\begin{proof}
Recall that the proof of the structure result of $\ti \cX$ carried out in \cite{HL1} is by deriving
the local defining equation of $\cX$, which was originally derived by Zinger using analytic method.
The derivation of such defining equation relies on solving the following resolution problem:
Let $S$ be a scheme and $p_S:\cC_S\to S$ be a flat family of nodal arithmetic genus one
curves together with an invertible sheaf $\sL_S$ on $\cC_S$ so that its restriction to
closed fibers $\cC_s$, $s\in S$, are generated by
global sections; one then finds an explicit complex of locally free sheaves of $\sO_S$-modules that is
quasi-isomorphic to $R^\bullet p_{S\ast} \sL_S$.

Once such explicit complex is found, using that deforming a stable map $u: C\to \Pf$
amounts to deforming the data $(C,L, u_1,\cdots, u_5)$ (up to equivalence) of a pair of a
line bundle on a curve and sections $u_i\in \Gamma(C,L)$, we obtain the local defining equation
of $\cX$ (cf. \cite{HL1}).

Deriving the resolution property of $R\bul p_{S\ast}\sL_S$ is an elementary problem in algebraic
geometry. It was done in \cite{HL1} as follows: Suppose
$S$ as before is a scheme; $s_0\in S$ is a closed point so that $\cC_{s_0}=E\cup R$ is a union of
a smooth elliptic curve and a smooth rational curve, and $\sL_S|_E\cong \sO_E$ and $\sL_S|_R$ is ample.
Let $\xi\in\Gamma(\sO_S)$ be such that $(\xi=0)$ is the locus where $\cC_S$ is nodal.
Then after shrinking $S$ if necessary, we have quasi-isomorphism of derived objects
$$R^\bullet p_{S\ast} \sL_S\cong_{q.i.} [\sO_S\mapright{\times \xi}\sO_S]\oplus [\sO_S^{\oplus r}\lra 0],
$$
where $r=\deg\sL|_{\cC_{s_0}}$.

We now derive the local defining equation of $\cY$. Since deforming a closed point
$([u,C],\psi)\in \cY$ is equivalent to deforming $(C,L,u_i,\psi)$ up to equivalence, where
$(C,L,u_i)$ is as before and $\psi\in\Gamma(C, L^{\otimes(-5)}\otimes\omega_C)$,
we need to have, for the family $(\cC_S\to S, \sL_S)$ as before, the resolution property of
$R^\bullet p_{S\ast} (\sL_S^{\oplus 5}\oplus \sL_S^{\otimes(-5)}\otimes\omega_{\cC_S/S})$.
Like the case for $R^\bullet p_{S\ast} \sL_S$, this case can be derived from the simple
case where $\cC_{s_0}=E\cup R$ and $\sL_S|_E\cong \sO_E$ as before.
In this case, by \cite{HL1},
$$R^\bullet p_{S\ast} \sL_S^{\otimes 5}\cong_{q.i.} [\sO_S\mapright{\times \xi} \sO_S]\oplus [\sO_S^{\oplus 5r}\lra 0].
$$
By Serre duality, we obtain
$$R^\bullet p_{S\ast} \sL_S^{\otimes (-5)}\otimes \omega_{\cC_S/S}
\cong_{q.i.} [\sO_S\mapright{\times \xi} \sO_S]\oplus [0\lra  \sO_S^{\oplus 5r}].
$$
Since it is a direct sum of $ [\sO_S\mapright{\times \xi} \sO_S]$ with a shift of a locally free sheaf,
the proof of \cite{HL1} can be adopted line by line to the case of $\cY$ to prove the statement of
the Lemma. Since the details are parallel, we will omit the details here.
\end{proof}

\begin{coro}\label{3.2}
All $\ti\cX\lalp$ and $\ti\cY\lalp$, where $\alpha=\mathrm{pri}$ or $\mu\vdash d$, are smooth;
the  tautological projection $\ti\cY\lalp\to\ti\cX\lalp$ is an isomorphism
(resp. an $\Ao$-bundle) for $\alpha=\mathrm{pri}$ (resp. $\alpha=\mu\vdash d$).
\end{coro}

\begin{lemm}\label{cone1}
Let $\bC_{\tcy/\ti\cD}\sub h^1/h^0(E_{\tcy/\ti\cD})$ be the intrinsic normal cone embedded via the obstruction
theory $\phi_{\tcy/\ti\cD}$ of $\tcy\to\ti\cD$.
\begin{enumerate}
\item Away from $\tcy\lgst$, it 
is the zero section of $h^1/h^0(E_{\tcy/\ti\cD})|_{\tcy-\tcy\lgst}$.
\item Away from $\tcy\lpri$, it is a rank two subbundle stack of $h^1/h^0(E_{\tcy/\ti\cD})|_{\tcy-\tcy\lpri}$.
\end{enumerate}
\end{lemm}

\begin{proof}
We consider the morphism
$$p: \tcy\lra\ti\cD.
$$
By the previous Proposition and Corollary, we know that both $\tcy-\tcy\lgst$
and the tautological morphism $\tcy-\tcy\lgst\to \cfM$ are smooth. 
Since the morphism
$f_\tcy: \cC_\tcy\to \Pf$ has positive degree along the (minimal) genus one subcurves of the fibers of $\cC_\tcy|_{\tcy-\tcy\lgst}\to \tcy-\tcy\lgst$,
one checks that $p|_{\tcy-\tcy\lgst}: \tcy-\tcy\lgst\to \ti\cD$ is also smooth. Therefore, $\bC_{\tcy/\ti\cD}|_{\tcy-\tcy\lgst}$ is
the zero section of $h^1/h^0(E_{\tcy/\ti\cD})|_{\tcy-\tcy\lgst}$.

For $\tcy-\tcy\lpri$, by the previous Proposition and Corollary, the tautological $\tcy-\tcy\lpri\to \cfM$ has its image a locally closed
Cartier divisor in $\cfM$, and $\tcy-\tcy\lpri$ is smooth onto this image.
Further, because the restriction of the tautological line bundle
$\sL_\tcy$ to the minimal genus one subcurves of the fibers of $\cC_\tcy|_{\tcy-\tcy\lpri}\to \tcy-\tcy\lpri$ are
generated by global sections and have degree zero, they are trivial line bundles. Thus the image
$p(\tcy-\tcy\lpri)\sub \ti\cD$ is a locally closed local complete intersection codimension two substack,
and $p|_{\tcy-\tcy\lpri}$ is smooth onto this image. Thus
$\bC_{\tcy/\ti\cD}|_{\tcy-\tcy\lpri}$ is a rank two bundle stack over $\tcy-\tcy\lpri$.
\end{proof}


By this Lemma, we know that to each $\alpha=\text{pri}$ of $\mu\vdash d$, the cone $\bC_{\tcy/\ti\cD}$ contains
a unique irreducible component (of multiplicity one) that dominates $\tcy_\alpha$. For the irreducible component
that dominates $\tcy\lpri$, it is
the closure of the zero section of $h^1/h^0(E_{\tcy/\ti\cD})|_{\tcy-\tcy\lgst}$ (in $h^1/h^0(E_{\tcy/\ti\cD})|_{\tcy\lpri}$);
we denote it by $\bC\lpri$. For $\mu\vdash d$, it is the closure of a rank two subbundle stack in
$h^1/h^0(E_{\tcy/\ti\cD})|_{\tcy_\mu}$ described in Lemma \ref{cone1}; we denote this subcone by $\bC_\mu'$.

There are possibly other irreducible components of $\bC_{\tcy/\ti\cD}$ lying over $\ti\Delta=\tcy\lpri\cap \tcy\lgst$.
We group them into $\sum_{\mu\vdash d} \bC_\mu\dpri$ such that $\bC_\mu\dpri$ lies over $\ti\Delta_\mu=\ti\Delta\times_\tcy\tcy_\mu$.
We comment that this grouping is not unique.
We write $\bC_\mu=\bC_\mu\pri\cup\bC_\mu\dpri$.

Therefore
\beq\label{C-dec}
[\bC_{\tcy/\ti\cD}]=[\bC\lpri]+\sum_{\mu\vdash d} [\bC_\mu]\in Z\lsta h^1/h^0(E_{\tcy/\ti\cD}).
\eeq
Consequently,  (denoting $[\bcgst]=\sum_{\mu\vdash d}[\bC_\mu]$,)
\beq\label{split-0}
[\ti\cY]\virt\lloc=0_{\ti\si,\loc}^![\bcpri]+0_{\ti\si,\loc}^![\bcgst]=0_{\ti\si,\loc}^![\bcpri]+\sum_{\mu\vdash d}0^!_{\ti\si,\loc}[\bC_\mu].
\eeq

Let $N_1(d)_Q\ured$ be the reduced genus one GW-invariants of the quintic $Q$ introduced in
\cite{LZ} (cf. Introduction).

\begin{prop}\label{prop1}
We have
$\deg 0_{\ti\si,\loc}^![\bcpri]=(-1)^{5d} N_1(d)_Q\ured$.
\end{prop}

Let $(d)$ be the partition of $d$ into a single part; i.e., the non-partition of $d$.

\begin{prop}\label{prop2}
For $\mu\ne (d)\vdash d$, we have
$\deg 0_{\ti\si,\loc}^![\bC_\mu]=0$. 
\end{prop}

\begin{prop}\label{prop3}
We have
$\deg 0_{\ti\si,\loc}^![\bC_{(d)}]= \frac{(-1)^{5d}}{12} N_0(d)_Q$.
\end{prop}

These three Propositions, the identity \eqref{split-0} and Proposition \ref{equal} combined
give an algebraic proof of the hyperplane property
of genus one GW-invariants of quintics proved originally via analytic method in \cite{LZ, VZ, Z}.

\section{Contribution from the primary component}

Let  $(f_\tcy,\pi_\tcy): \cC_\tcy\to \Pf\times\tcy$ with $\psi_\tcy\in\Gamma(\cC_\tcy, \sP_\tcy)$
be the tautological family of $\tcy$, where $\sL_\tcy=f_\tcy\sta\sO(1)$ and $\sP_\tcy=\sL_\tcy^{\otimes(-5)}\otimes
\omega_{\cC_\tcy/\tcy}$. Recall that (the deformation complex of the relative obstruction theory)
$E_{\tcy/\ti\cD}=R\bul\pi_{\tcy\ast}(\sL_\tcy^{\oplus 5}\oplus \sP_\tcy)$.
We let
$$\bH_1=h^1/h^0\bl (R\bul\pi_{\tcy\ast} \sL_\tcy^{\oplus 5})|_{\tcy\lpri}\br
\and \bH_2=h^1/h^0\bl (R\bul\pi_{\tcy\ast} \sP_\tcy)|_{\tcy\lpri}\br.
$$
By the base change property of the $h^1/h^0$-construction, we have
$$h^1/h^0(E_{\tcy/\ti\cD})\times_\tcy\tcy\lpri=\bH_1\times_{\tcy\lpri}\bH_2.
$$

Now let $\tcy\lpri^{\circ}=\tcy\lpri-\ti\Delta$. where $\ti\Delta=\tcy\lgst\cap\tcy\lpri$. By its definition
and an easy vanishing argument,
we have $(R^1\pi_{\tcy\ast} \sL_\tcy^{\oplus 5})|_{\tcy\lpri^{\circ}}=0$,
and $(R^1\pi_{\tcy\ast} \sP_\tcy)|_{\tcy\lpri^{\circ}}\cong ((\pi_{\tcy\ast}f_\tcy\sta\sO(5))|_{\tcy\lpri^{\circ}})\dual$,
which is locally free of rank $5d$  over $\tcy\lpri^{\circ}$. We let $0_{\bH_2}\sub \bH_2$ be the closure of
the zero section of $\bH_2|_{\tcy\lpri^{\circ}}$ (in $\bH_2$).

\begin{lemm}\label{main-Cone}
Let $\iota: \bH_1\times_{\tcy\lpri}\bH_2\to h^1/h^0(E_{\tcy/\ti\cD})$
be the inclusion, then we have
$$[\bcpri]=\iota\lsta[\bH_1\times_{\tcy\lpri} 0_{\bH_2} ]\in Z\lsta h^1/h^0(E_{\tcy/\ti\cD}).
$$
\end{lemm}

\begin{proof}
First, over $\tcy\lpri^\circ$, $\bH_1$ is the zero bundle-stack. Thus by Lemma \ref{cone1} and the discussion afterwards,
the identity stated in the Lemma
holds when restricted
to $\tcy\lpri^{\circ}$. By their constructions, both $\bC\lpri$ and $\bH_1\times_{\tcy\lpri} 0_{\bH_2}$ are irreducible. Thus the identity holds.
\end{proof}

\begin{coro}\label{4.2}
We have $0^!_{\ti\sigma,\loc}([\bcpri])=0^!([0_{\bH_2}])$,
where the first is taken under the localized Gysin map of $A\lsta h^1/h^0(E_{\tcy/\ti\cD})_{\ti\si}$, and the second
is taken under the ordinary Gysin map of $\bH_2$.
\end{coro}

\begin{proof}
Since $\ti\cY\lell$ is proper, and since the cosection localized Gysin map is the same
as the ordinary Gysin map over proper bases, the proposition follows from that $0^![\bH_1\times_{\tcy\lpri} 0_{\bH_2}]=0^![0_{\bH_2}]$,
where the first $0^!$ is taken in $\bH_1\times_{\tcy\lpri} \bH_2$ and
second $0^!$ is taken in $\bH_2$.
\end{proof}

To proceed, we prove a useful result.

\begin{lemm}\lab{HL-euler}
Let $R=[\sR_0\to \sR_1]$ be a complex of locally free sheaves on an integral Deligne-Mumford stack $M$ such that
$H^1(R)$ is a torsion sheaf on $M$ and the image sheaf of $\sR_0\to \sR_1$ is locally free.
Let $U\sub M$ be the complement of the support of $H^1(R)$, and let $\bB\sub h^1/h^0(R\dual[-1])$
be the closure of the zero section  of the vector bundle $h^1/h^0(R\dual[-1]|_U)= H^0(R|_U)\dual$. Then
$$0^![\bB]=e(H^0(R)\dual)\in A\lsta M.
$$
\end{lemm}

\begin{proof}
Let $\sK_i=H^i(R)$, which fit into the exact sequence
\beq\label{KRRK}0\lra \sK_0\mapright{\alpha} \sR_0\mapright{\beta} \sR_1\mapright{\gamma} \sK_1\lra 0.
\eeq
(Since $\sK_0$ and $\sR_i$ are locally free, we will use the same symbol to denote its associated
vector bundle.)
By assumption and the choice of $U$, $\beta\dual|_U$  is a sub-bundle homomorphism.
Hence $[\sR_0\dual/\sR_1\dual]|_{U}$ is a vector bundle over $U$. Viewing $U\sub [\sR_0\dual/\sR_1\dual]|_{U}$ as
the zero-section, its closure in the bundle stack $h^1/h^0(R\dual[-1]|_U)=[\sR_0\dual/\sR_1\dual]$
is the $\bB$ referred to in the statement of the Lemma.

Let $b:\sR_0\dual\lra [\sR_0\dual/\sR_1\dual]$ be the quotient. As
$\beta\dual(\sR_1\dual|_U)\sub \sR_0\dual|_{U}$ is a sub-bundle; we let $\bD$ be the closure of $\beta\dual(\sR_1\dual|_U)$
in $\sR_0\dual$. The smoothness of $b$ implies $b^{-1}(\bB)=\bD$.
Hence
$$0^!([\bB])=0^!([\bD])\in A\lsta M,
$$
where the two Gysin maps are intersecting with the zero sections of $[\sR_0\dual/\sR_1\dual]$ and of
$\sR_0\dual$, respectively.

We are left to show $0^!([\bD])=e(\sK_0\dual)$. (By assumption, $\sK_0$ is locally free.)
Since $\ker \gamma$ is locally free,  $\alpha\dual:\sR_0\dual\to \sK_0\dual$ is a surjection of vector bundles.
On the other hand since $\sK_1|_U=0$ we can restrict (\ref{KRRK}) on $U$ and take its dual to get an exact sequence of locally free sheaves
$$0\lra \sR_1\dual|_{U} \mapright{\beta\dual} \sR_0\dual|_U\mapright{\alpha\dual} \sK_0\dual|_U\lra 0.
$$
Hence the closure of  $\beta\dual(\sR_1\dual|_{U})$ in  $\sR_0\dual$ equals the closure of $\ker(\alpha\dual)|_U$ in $\sR_0\dual$;
thus $\bD=\ker\alpha\dual$.
As $\sR_0\dual\to \sK_0\dual$ is a surjective bundle homomorphism, we have
$0^!([\bD])=e(\sK_0\dual)$.
\end{proof}

Let $(f_{\tcy\lpri},\pi_{\tcy\lpri})$ be the restriction of $(f_\tcy,\pi_\tcy)$ to $\tcy\lpri$.  Let
$$R=(R\bul\pi_{\tcy\ast} f_\tcy^\ast \sO_\Pf(5))|_{\tcy\lpri}=R\bul\pi_{\tcy\lpri\ast} f_{\tcy\lpri}^\ast \sO_\Pf(5).
$$
By Serre duality, $R\dual\cong (R\bul\pi_{\tcy\ast} f_\tcy^\ast \sP_\tcy)|_{\tcy\lpri}$.
Like in \cite{Beh, LT}, we can represent $R\bul\pi_{\tcY\lpri\ast} f\sta_{\ti\cY\lpri}\sO_{\Pf}(5)=[\sR_0\to\sR_1]$
as a complex of locally free sheaves.
On the other hand, by \cite{VZ} (see also \cite{HL1}, or Prop. \ref{coordinate}), $\pi_{\tcy\lpri\ast}f_{\tcy\lpri}\sta\sO_\Pf(5)$ is locally free of rank $5d$,
and $R^1\pi_{\tcy\lpri\ast}f_{\tcy\lpri}\sta\sO(5)$ is torsion.
Thus applying Corollary \ref{4.2} and Lemma \ref{HL-euler}, we obtain

\begin{coro}\label{LZ}
We have identity
$$\deg\, 0_{\ti\si,\loc}^![\bcpri]=c_{5d}\bl ( \pi_{\tcy\lpri\ast}f_{\tcy\lpri}\sta\sO_\Pf(5))\dual\br=
(-1)^{5d} c_{5d}\bl \pi_{\tcy\lpri\ast}f_{\tcy\lpri}\sta\sO_\Pf(5)\br.
$$
\end{coro}

Note that the right hand side is the reduced genus one GW-invariants introduced in \cite{LZ}.

\section{Cosection localized Gysin map via compactifications}
\def\barW{{\overline W}}
\def\barV{{\overline V}}
\def\barC{\overline C}

We begin with a general discussion of a special type of localized Gysin maps.
Let $M$ be a Deligne-Mumford stack, $L$ be a line bundle on $M$, and
$V_1$ and $V_2$ be two vector bundles on $M$. Suppose we are given two
vector bundle homomorphisms\footnote{Bundle homomoprhisms in this paper are
possibly degenerate, (i.e. those induced by sheaf-homomorphisms
of respective bundles).}
\beq\label{xixi} \xi_1: V_1\otimes L \lra \sO_M,\and
\xi_2: V_2\to\sO_M
\eeq
such that $\xi_1$ is surjective. (Here $\sO_M$ is viewed as the rank one trivial line bundle on $M$.)

We let $\gamma: W=\text{Total}(L)\to M$ be the total space of $L$. The tautological (identity) section
$\epsilon\in \Gamma(W, \gamma\sta L)$ paired with $\gamma\sta(\xi_1)$ defines a bundle homomorphism
\beq\label{xi2}
\ti\xi_1\defeq\gamma\sta(\xi_1)(\cdot \otimes \epsilon): \ti V_1\defeq \gamma\sta V_1\lra \sO_W.
\eeq
We let $\ti\xi_2=\gamma\sta(\xi_2): \ti V_2\defeq  \gamma\sta V_2\to\sO_W$.

Let
$$\ti\xi=(\ti \xi_1,\ti\xi_2): \ti V\defeq \ti V_1\oplus \ti V_2\lra \sO_W.
$$
Since $\xi_1$ is surjective, $\ti\xi_1$ is surjective away from the zero-section $0_W= M\sub W$;
thus the non-surjective locus $D(\ti\xi)$ of $\ti\xi$ is contained in
the zero-section $0_W\sub W$. We let $U=W-D(\ti\xi)$, and form
$$\ti V(\ti\xi)=\ti V|_{D(\ti\xi)}\cup \ker\{ \ti\xi\,|_{U}: \ti V|_{U}\to \sO_{U}\}\sub \ti V.
$$
The cosection $\ti\xi$ defines a localized Gysin map (cf. \cite{KL})
\beq\label{ti-xi} 0^!_{\ti\xi,\loc}: A\lsta \ti V(\ti\xi)\lra A\lsta D(\ti\xi).
\eeq

This localized Gysin map has a simple interpretation in the situation under consideration.
We let
$$
\bar\gamma: \barW=\bP(L\oplus \sO_M)\lra M
$$
be the obvious compactification of $W=\barW-D_\infty$, where $D_\infty=\bP(L\oplus 0)$ is a divisor of $\barW$. We extend
$\ti V_1$ and $\ti V_2$ to $\barW$ via
$$\barV_1=\bar\gamma\sta V_1(-D_\infty)\and \barV_2=\bar\gamma\sta V_2.
$$
We let $\bar\xi_2=\bar\gamma\sta \xi_2$, which is the extension of $\ti\xi_2$.
Because of the expression \eqref{xixi}, $\ti\xi_1$ extends to a homomorphism
$\bar\xi_1: \barV_1=\bar\gamma\sta (-D_\infty)\to \sO_{\barW}$. Let
$$\bar\xi=(\bar\xi_1,\bar\xi_2): \barV\defeq \barV_1 \oplus \barV_2\lra \sO_\barW.
$$
Because $\ti\xi_1$ has the form \eqref{xi2} and $\xi_1$ is surjective, $\bar\xi_1$ is surjective along $D_\infty$;
thus so does $\bar\xi$. 
Consequently, the non-surjective locus of $\bar\xi$ and $\ti\xi$ are identical; namely
\beq\label{deg-locus}
D(\bar\xi)=D(\ti\xi).
\eeq

\begin{lemm}
Let $\iota_!: Z\lsta \ti V(\ti\xi)\to Z\lsta \barV$ be defined by sending a closed integral $[C]\in Z\lsta \ti V(\ti\xi)$
to the cycle of its closure in $\barV$: $\iota_![C]=[\overline C]\in Z\lsta \barV$,
and then extending to $Z\lsta \ti V(\ti\xi)$ by linearity. Let $\tau: D(\ti\xi)\to M$ be the inclusion. Then we have
$$\bar\gamma\lsta\circ 0^!_{\barV}\circ \iota_! = \tau\lsta\circ 0^!_{\ti\xi,\loc}: Z\lsta \ti V(\ti\xi)\lra A\lsta M.
$$
\end{lemm}

\begin{proof}
We let $\barV(\bar\xi)$ be defined similar to $\ti V(\ti\xi)$ with `` $\ti{}$ '' replaced by `` $\bar{}$ ".
Let $0^!_{\bar\xi,\loc}: Z\lsta \barV(\bar\xi)\to A\lsta D(\bar\xi)$ be the localized Gysin map associated to $\bar\xi$.
Let $\tau': \barV(\bar\xi)\to \bar V$ and $\tau\dpri: D(\bar\xi)\to \barW$ be the inclusions. By \cite[Prop 1.3]{KL}, we have the commutative square
$$\begin{CD}
Z\lsta \barV(\bar \xi) @>{0^!_{\bar\xi,\loc}}>> A\lsta D(\bar\xi)\\
@VV{\tau'\lsta}V @VV{\tau\lsta\dpri}V\\
Z\lsta \barV @>{0^!_{\barV}}>> A\lsta \barW.
\end{CD}
$$
On the other hand, since $\bar\xi$ is an extension of $\ti\xi$, the homomorphism
$\iota_!: Z\lsta \ti V(\ti\xi)\to Z\lsta \barV$ factors through $\iota_!': Z\lsta \ti V(\ti\xi)\to Z\lsta \barV (\bar \xi)$.
Composing, we obtain
$$\tau\lsta\dpri\circ 0^!_{\bar\xi,\loc}\circ\iota_!'=0^!_{\barV}\circ \tau\lsta'\circ\iota_!'=0^!_{\barV}\circ\iota_!.
$$
Since $\bar\xi$ is an extension of $\ti\xi$ and $D(\bar\xi)=D(\ti\xi)$, tracing through the construction of the localized
Gysin maps, we conclude $0^!_{\bar\xi,\loc}\circ\iota_!'=0^!_{\ti\xi,\loc}$. Composed with $\bar\gamma\lsta :A\lsta \barW\to A\lsta M$, we obtain
$$\bar\gamma\lsta\circ 0^!_\barV\circ\iota_!=\bar\gamma\lsta\circ   \tau\lsta\dpri\circ 0^!_{\bar\xi,\loc}\circ\iota_!' =\tau\lsta\circ 0^!_{\ti\xi,\loc}:
Z\lsta \ti V(\ti\xi)\lra A\lsta M.
$$
This proves the Lemma.
\end{proof}

\begin{coro}\label{count}
Let $C\sub \ti V(\ti\xi)$ be closed and integral; let $\overline C\sub \barV$ be its closure, and
let $\overline C_b=\overline C\cap (0\oplus \barV_2)$. We let $N_{\overline C_b}\overline C$
be the normal cone to $\overline C_b$ in $\overline C$; it is a cycle in $\barV$.
Then we have
\beq\label{n-cone}
\deg 0^!_{\ti\xi,\loc}[C]=\deg 0^!_{\barV}[\overline C]=\deg 0^!_{\barV} 
[N_{\overline C_b}\overline C].
\eeq
\end{coro}

The intersection $\barC_b$ has an easy description in the homogeneous case.
Since $W$ is the total space of the line bundle $L$,
we have the dilation morphism
\beq\label{mt}
\fm_t: W\lra  W
\eeq
defined by sending $x\in L|_{x'}$ (over $x'\in M$) to
$t x\in L|_{x'}\sub W$. This defines a $\CC\sta$-action on $W$ that preserves the morphism $W\to M$
with $\CC\sta$ acting trivially on $M$; its fixed locus is the zero section $0_W\sub W$.

We define
$\Phi_{i,0}(t): \ti V_i\lra \fm_t\sta \ti V_i$
to be the homomorphism that keeps $\gamma\sta e$ invariant, where $\gamma: W\to M$ and $\gamma^\ast e\in V_i$; and for $k\in\ZZ$, we define
\beq\label{Phi12}
\Phi_{i,k}(t)=t^k\cdot \Phi_{i,0}(t): \ti V_i\lra \fm_t\sta \ti V_i.
\eeq

\begin{defi}\label{def-weight}
We say a closed integral $C\sub \ti V$ homogeneous of weight $(k_1,k_2)$ if it is invariant under $(\Phi_{1,k_1}(t), \Phi_{2,k_2}(t))$ for
all $t\in \CC\sta$. We say a cycle $\alpha\in Z\lsta \ti V$ is homogeneous of weight $(k_1,k_2)$ if
each of its integral components is homogeneous of weight $(k_1,k_2)$.
\end{defi}

We now investigate the intersection $\barC_b=\barC\cap (0\oplus \barV_2)$ near $D_\infty$. Since this is a
local problem, by replacing $M$ by its affine \'etale chart, we can assume $M=\spec A$ is affine,
and $L\cong \sO_M$ and $V_i\cong \sO_M^{\oplus n_i}$ are trivialized.
Using such trivializations, we have induced isomorphisms
\beq\label{coor}
\barW=M\times\Po,\quad  \barV_2\cong (M\times\Po)\times\A^{n_2},
\eeq
and $D_\infty=M\times\{\infty\}\sub M\times\Po$.

\begin{lemm}\label{compactify}
Let $M=\spec A$, and $L$ and $V_i$ are trivialized as stated.
Let $C\sub \ti V$ be a homogeneous closed integral substack of weight $(0,1)$. Then
under \eqref{coor}, and for a $B\sub M\times\A^{n_2}$,
$$\barC\cap(0\oplus \barV_2)|_{\barW-0_W}=
B\times(\Po-0).
$$
\end{lemm}

\begin{proof}
Adding $V_1\cong \sO_M^{\oplus n_1}$, we have the induced isomorphism
$$\ti V\cong (M\times\Po)\times \A^{n_1}\times\A^{n_2}.
$$
We let $t$ (resp. $x=(x_i)$; resp. $y=(y_j)$) be the standard coordinate variables of $\Ao=\Po-\{\infty\}$
(resp. of $\A^{n_1}$; resp. of $\A^{n_2}$). Because $C$ is homogeneous of weight $(0,1)$,
the ideal of $C|_{W-0_W}\sub \barV|_{W-0_W}$
is generated by elements
$$\{f(x_1,\cdots,x_{n_1},  t\upmo y_1,\cdots,  t\upmo y_{n_2})\mid  f\in J\sub A[x, y]\},
$$
where $J$ is an ideal in the polynomial ring $A[x,y]=A[x_1,\cdots, x_{n_1}, y_1,\cdots,y_{n_2}]$.

We now pick a new trivialization of $\barV_1$ over $\barW-0_W$.
We let $\eps_1,\cdots,\eps_{n_1}$ be a basis of $\barV_1$ that is the pullback of a
basis of $V_1$ (over $M$). As $t\upmo$ extends to a regular function near $D_\infty$ and has vanishing
order one along $D_\infty$, $e_i=t\upmo\cdot \eps_i$, $1\le i\le n_1$,  form a basis of $\barV_1|_{\barW-0_W}$.

%
%
%
We let
\beq\label{new-iso}
\barV|_{\barW-0_W}\cong M\times(\Po-0)\times\A^{n_1}\times\A^{n_2}
\eeq
be the isomorphism induced by the given trivializations of $L$ and $V_2$,
and the new trivialization of $\barV_1$ using the basis $e_i$.
We let $x'=(x_i')$ be the standard coordinate variables of
$\A^{n_1}$ in \eqref{new-iso}, then $x_i$ and $x_i'$ are related by
$x_i=t\upmo x_i'$.
Thus in the coordinates $(x_i', y_j)$, $C|_{W-0_W}$ is defined by the ideal generated by
\beq\label{new-gen}\{f(t\upmo x_1',\cdots, t\upmo x_{n_1}',  t\upmo y_1,\cdots,  t\upmo y_{n_2})\mid  f\in J\sub A[x, y]\}.
\eeq
Finally, since $C\sub \ti V$ is a cone, $J$ is a homogeneous ideal; thus the same ideal
is also generated by
\beq\label{new-2}
\{f(x_1',\cdots,x_{n_1}',  y_1,\cdots,  y_{n_2})\mid  f\in J\sub A[x, y]\}.
\eeq

Therefore, 
$\barC\cap \barV|_{\barW-0_W}\sub \barV|_{\barW-0_W}$
is defined by the ideal generated by \eqref{new-2}. Thus $\barC\cap \barV|_{\barW-0_W}=
B'\times(\Po-0)$ for a $B'\sub M\times \A^{n_1}\times\A^{n_2}$. 
Intersecting with
$\barV_2|_{\barW-0_W}$ proves the Lemma.
\end{proof}

We will show that the pair $\tcX\lgst\sub \tcy\lgst$ fits into the description of the pair
$M\sub W$ described.
To this end, we introduce the corresponding vector bundles $V_i$ and $\ti V_i$. As the deformation complex of $\tcY/\tcD$ is $E_{\tcy/\ti\cD}=R\bul\pi_{\tcy\ast}(\sL_\tcy^{\oplus 5}\oplus \sP_\tcy)$, we introduce
$$\ti\bV_1=h^1/h^0((R\bul\pi_{\tcy\ast}\sL_{\tcy}^{\oplus 5})|_{\tcy\lgst}),\ \
\ti\bV_2=h^1/h^0((R\bul\pi_{\tcy\ast}\sP_{\tcy})|_{\tcy\lgst}),\ \  \ti\bV=\ti\bV_1\times_{\tcy\lgst}\ti\bV_2.
$$
By the base change property of the $h^1/h^0$-construction, and the construction of $\bC\lgst$ (cf. \eqref{split-0}), 
$$[\bcgst]\in Z\lsta \ti\bV;\quad \ti\bV=h^1/h^0(E_{\tcy/\ti\cD})|_{\tcy\lgst}.
$$

We let
\beq\label{tiV-2}
\ti V_1=H^1((R\bul\pi_{\tcy\ast}\sL_{\tcy}^{\oplus 5})|_{\tcy\lgst}),\ \
\ti V_2=H^1((R\bul\pi_{\tcy\ast}\sP_{\tcy})|_{\tcy\lgst}), \ \text{and}\ \  \ti V=\ti V_1\oplus \ti V_2.
\eeq
Since both $\ti V_1$ and $\ti V_2$ are locally free on
$\tcy\lgst$, by abuse of notation, we will view them as vector bundles as well.
Then using the arrow $(R\bul\pi_{\tcy\ast}\sL_{\tcy}^{\oplus 5})|_{\tcy\lgst}\to H^1((R\bul\pi_{\tcy\ast}\sL_{\tcy}^{\oplus 5})|_{\tcy\lgst})$,
we obtain a canonical morphism $\rho_1: \ti\bV_1\to \ti V_1$; similarly, we have a canonical $\rho_2:\ti\bV_2\to \ti V_2$.
Note that both $\rho_i$ are proper, and $\ti V_i$ are (relative to $\tcy\lgst$) coarse moduli spaces of $\ti\bV_i$.
By the definition of $\ti V_i$, the sheaf of sections of $\ti V$ is $\Ob_{\tcy/\ti\cD}|_\tcygst$.
We denote by $\ti\xi$ the restriction
\beq\label{xi-2}
\ti\xi\defeq \ti\si|_\tcygst: \ti V\lra \sO_\tcygst,
\eeq
where $\ti\si$ is given in \eqref{ti-si}.

We now justify the symbols for the pair $M\sub W$.
Let $(f_{\tcX\lgst}, \pi_{\tcX\lgst}): \cC_{\tcX\lgst}\to \Pf\times {\tcX\lgst}$ be the pullback of the universal family of
$\barM_1(\Pf,d)$. Let $\sL_{\tcX\lgst}=f_{\tcX\lgst}\sta\sO(1)$ and
$\sP_{\tcX\lgst}=\sL_{\tcX\lgst}^{\otimes(-5)}\otimes \omega_{\cC_{\tcX\lgst}/\tcX\lgst}$.
Then by \cite{HL1}, $\tcy\lgst$ is the total space of a line bundle on $\tcX\lgst$:
$$L=(\pi_{\tcX\lgst})\lsta \sP_{\tcX\lgst}.
$$
Thus letting $M=\tcX\lgst$ and $W=\tcy\lgst$ fit into the description $W=\text{Total}(L)$ given
before.

We next introduce $V_i$ and $\xi_i$ on $M$ that fit the whole package as well:
\beq\label{VV}
V_1=R^1\pi_{\tcX\lgst\ast}\sL_{\tcX\lgst}^{\oplus 5},\quad
V_2=R^1\pi_{\tcX\lgst\ast}\sP_{\tcX\lgst},\and  V=V_1\oplus V_2.
\eeq
They are vector bundles (locally free sheaves) on $\tcX\lgst$($=M$).
Let
$$\gamma: W\defeq \tcy\lgst=\Tot(L)\lra  \tcX\lgst=M
$$
be the induced (tautological) projection.

\begin{lemm}
We have canonical isomorphisms $\gamma\sta V_i=\ti V_i$; further
there are $\xi_1: V_1\otimes L \to \sO_M$ and $\xi_2: V_2\to \sO_M$ so that
$\ti\xi$ in \eqref{xi-2} is identical to the $\ti\xi$ constructed from $(\xi_1,\xi_2)$ via the construction
given in \eqref{xi2} and after.
\end{lemm}

\begin{proof}
We let
$f_{\tcy\lgst}, \pi_{\tcy\lgst}, \cC_{\tcy\lgst}, \sL_{\tcy\lgst}$ and $\sP_{\tcy\lgst}$ be objects over
$\tcy\lgst$ defined similarly as that over $\tcX\lgst$. Then the $\ti V_1$ and $\ti V_2$ in \eqref{tiV-2}
are
$\ti V_1=R^1\pi_{\tcy\lgst\ast}\sL_{\tcy\lgst}^{\oplus 5}$ and $\ti V_2=R^1\pi_{\tcy\lgst\ast}\sP_{\tcy\lgst}$.
On the other hand, since we have canonical isomorphism
$\cC_{\tcy\lgst}\cong \cC_{\tcX\lgst}\times_{\tcX\lgst}\tcy\lgst$, letting $\ti\gamma: \cC_{\tcy\lgst}\to \cC_{\tcX\lgst}$
be the induced projection, then we obtain $\ti\gamma\sta\sL_{\tcX\lgst}= \sL_{\tcy\lgst}$,
and same for $\sP_{\tcy\lgst}$. Applying the base change formula, (since $R^2\pi_{\tcX\lgst\ast}=0$,)
we obtain the canonical isomorphisms
$$\ti V_1= R^1\pi_{\tcy\lgst\ast}\sL_{\tcy\lgst}^{\oplus 5}
= R^1\pi_{\tcy\lgst\ast}\ti\gamma\sta \sL_{\tcX\lgst}^{\oplus 5}
\cong \gamma\sta R^1\pi_{\tcX\lgst\ast}\sL_{\tcX\lgst}^{\oplus 5}=\gamma\sta V_1;
$$
the same reason gives a canonical $\ti V_2\cong\gamma\sta V_2$.

For the $\ti\xi$ defined in \eqref{xi-2}, we recall the construction of the cosection $\si$ in \eqref{si0}.
For any closed $y\in\tcy\lgst$ over $x\in \tcX\lgst$ associated to $[u,C]$ with $u=(u_i)\in H^0(u\sta\sO(1)^{\oplus 5})$ and
$\psi\in H^0(u\sta\sO(-5)\otimes\omega_C)$, $\ti V_1|_y=H^1(u\sta\sO(1)^{\oplus 5})$,
$\ti V_2|_y=H^1(u\sta\sO(-5)\otimes\omega_C)$, and
$$\ti\si(y): \ti V_1|_y\oplus \ti V_2|_y\lra \CC;\quad  (\dot u_i, \dot \psi)\mapsto 5\psi\sum_{i=1}^5 u_i^4 \dot u_i+\dot\psi\sum_{i=1}^5 u_i^5.
$$
Accordingly, we define
$\xi_1: V_1\otimes L\to\sO_{\tcX\lgst}$ and $\xi_2: V_2\to\sO_{\tcX\lgst}$ over $x\in \tcX\lgst$ to be
$$\si_1(x)((\dot u_i)\otimes\psi)=5\psi\sum u_i^4 \dot u_i,\and \xi_2(x)(\dot\psi)=\dot\psi\sum u_i^5.
$$
With this definition, by \cite[Prop 3.4]{CL}, $\xi_1$ is surjective. As $\ti\xi$ is constructed following
\eqref{xi-2} that is the restriction of $\ti\si$ over $\tcy\lgst$, this proves the Lemma.
\end{proof}
%

In the following, we will view $V_i$, $\ti V_i$ and $\ti\xi$ be defined in \eqref{tiV-2}, \eqref{xi-2} and \eqref{VV},
which also satisfy the properties specified at and before \eqref{xi2}.

Following \cite{CL}, the non-surjective locus
$D(\ti\xi)$  of $\ti\xi=\ti\si|_{\tcy\lgst}$ is
$$D(\ti\si)\times_{\tcX}\tcX\lgst=\barM_1(Q,d)\times_{\barM_1(\Pf,d)}\tcX\lgst,
$$
which is proper.
As before, we form the substack
$\ti V(\ti\xi)\sub \ti V$; we let $C\lgst$ be the coarse moduli of $\bC\lgst$ relative to $\ti\cY\lgst$,
thus $C\lgst\sub \ti V$ since $\ti V$ is the coarse moduli of
$\ti\bV$.  Further, since the projection $\rho\defeq \rho_1\times\rho_2: \ti\bV\to \ti V$
is smooth, we have an identity of cycles
$\rho^\ast[C\lgst]= [\bcgst]\in Z\lsta \ti \bV$. Finally, because $[\bC\lgst]\in Z\lsta \ti\bV(\ti\si)$ (cf. \cite{CL}), we have
$$[C\lgst]\in Z\lsta \ti V(\ti\xi).
$$
As before,
$0^!_{\ti \xi,\loc}: A\lsta \ti V(\ti \xi) \to A\lsta D(\ti \xi)$
is the localized Gysin map.

\begin{prop}
We have the  identity
$$0^!_{\ti\si,\loc}[\bcgst]=0^!_{\ti \xi,\loc}[C\lgst]\in A\lsta D(\ti \xi). 
$$
\end{prop}

\begin{proof}
We can find a vector bundle $F$ on $\tcygst$ and a surjective bundle map $\phi: F\to \ti\bV$.
Let $\ti\si_\phi: F\to\sO_{\tcygst}$ be the pullback cosection. Then
$D(\ti\si_\phi)=D(\ti \xi)$, $\phi\sta[\bcgst]\in F(\ti\si_\phi)$,
and
$$0^!_{\ti\si_\phi,\loc}(\phi\sta[\bcgst])=0^!_{\ti\si,\loc}[\bcgst]\in A\lsta D(\ti \xi).
$$
On the other hand, let $\bar\phi: F\to \ti V$ be the composition of $\phi$ with $\rho$, then
$\ti\si_\phi=\bar\phi\sta(\ti \xi)$ and $\bar\phi\sta[C\lgst]=\phi\sta[\bcgst]$.
Thus
$$0^!_{\ti\si_\phi,\loc}(\phi\sta[\bcgst])=0^!_{\bar\phi\sta(\ti \xi),\loc}(\bar\phi\sta[C\lgst])
=0^!_{\ti \xi,\loc}[C\lgst] \in A\lsta D(\ti \xi).
$$
This proves the Proposition.
\end{proof}

Finally, we verify

\begin{lemm}\label{5.7}
The substack $C\lgst\sub \ti V$ is homogeneous of weight $(0,1)$.
\end{lemm}

\begin{proof}
As the obstruction theory of  $\tcy\to\ti\cD$ is the pullback of that of $\cY\to\cD$, and the later is
via the open $\cD$-embedding (cf. \eqref{dir-C})
\beq\label{j}\jmath: \cY=\barM_1(\Pf,d)^p\mapright{\sub} \fS\defeq C(\pi\lsta(\sL^{\oplus 5}\oplus \sP)),
\eeq
we will prove a corresponding result for $\fS$.

Let $\pi_\fS: \cC_\fS\to \fS$, $\sL_\fS$ and $\sP_\fS$ be the pullback of $(\pi: \cC\to\cD, \sL, \sP)$.
The modular construction of $\fS$ provides us a universal section
\beq\label{u-section} (u_{\fS,i},\psi_{\fS})
\in \Gamma(\cC_\fS, \sL_\fS^{\oplus 5}\oplus \sP_\fS)=
\Gamma(\fS, \pi_{\fS\ast}(\sL_\fS^{\oplus 5}\oplus \sP_\fS)), \quad (i=1,\cdots, 5).
\eeq
Namely, each closed $\xi=(\cC_\xi,u_{\xi,i},\psi_\xi)\in\fS$ has $u_{\fS,i}|_\xi\equiv u_{\xi,i}$ and
$\psi_\fS|_\xi\equiv \psi_\xi$.

For $t\in\CC\sta$, we introduce an automorphism
$$\fa_t: \sL^{\oplus 5}\oplus \sP\lra \sL^{\oplus 5}\oplus \sP
$$
that is $\id_{\sL^{\oplus 5}}$ when restricted to the summand $\sL^{\oplus 5}$, and is $t\cdot \id_{\sP}$ when restricted to
the summand $\sP$.  Clearly, $\fa_t$ commutes with the arrows of $\cD$,  and thus is a well-defined automorphism.
The collection $\{\fa_t\mid t\in\CC\sta\}$ defines a $\CC\sta$-action on $\sL^{\oplus 5}\oplus \sP$,
where $\CC\sta$ acts trivially on $\cD$.

Following the moduli construction of $\fS$, the automorphism $\fa_t$ induces a $\cD$-automorphism $\ti\fa_t: \fS\to\fS$: it
is defined by sending $(\cC_\tau,\sL_\tau, u_i,\psi)\in \fS(T)$, where $T$ is any scheme, to $(\cC_\tau,\sL_\tau, u_i,\psi)^t=(\cC_\tau,\sL_\tau, u_i,t\cdot \psi)$.
The collection $\{\ti\fa_t\mid t\in\CC\sta\}$ forms a $\CC\sta$-action on $\fS$; under this action the tautological projection
$\fS\to\cD$ is $\CC\sta$-equivariant.

Because $\cC_\fS=\cC\times_\cD\fS$, and because $\fS\to\cD$ is $\CC\sta$-equivariant (via $\ti\fa_t$ on $\fS$
and the trivial $\CC\sta$-action on $\cD$), the trivial $\CC\sta$-action on $\cC$ and the $\ti\fa_t$ on $\fS$ lifts to
a $\CC\sta$-action on $\cC_\fS\to\fS$. We denote this action by $\varphi_0(t): \cC_\fS\to \ti\fa_t\sta  \cC_\fS$.
Then since $\sL_\fS$ and $\sP_\fS$ are pullback of $\sL$ and $\sP$ on $\cC$, $\sL_\fS$ and $\sP_\fS$
admit the obvious $\CC\sta$-linearizations
\beq\label{action0}
 \varphi_1(t): \sL_\fS\lra  \varphi_0(t)\sta  \sL_\fS,\quad
\varphi_2(t): \sP_\fS\lra  \varphi_0(t)\sta  \sP_\fS
\eeq
characterized by that their $\CC\sta$-equivariant sections are pullback sections of $\sL$ and $\sP$, respectively.


The action $\ti\fa_t$ induces a tautological $\CC\sta$-action (linearlization)
\beq\label{action1}
\bar\varphi_1(t): \sL_\fS\to  \varphi_0(t)\sta  \sL_\fS,\quad
\bar\varphi_2(t): \sP_\fS\to  \varphi_0(t)\sta  \sP_\fS
\eeq
characterized by that the universal section \eqref{u-section} is $\CC\sta$-equivariant under this
linearization; i.e.,
\beq\label{action-eq}
(\bar\varphi_1(t)(u_{\fS,i}), \bar\varphi_2(t)(\psi_{\fS}))=(\varphi_0(t)\sta(u_{\fS,i}),\varphi_0(t)\sta(\psi_{\fS})).
\eeq
By the construction of $\ti\fa_t$ and the two $\CC\sta$-actions, we see that
$$\bar\varphi_1(t)=\varphi_1(t)\and \bar\varphi_2(t)=t\cdot \varphi_2(t).
$$

Let $\phi_{\fS/\cD}: L\bul_{\fS/\cD}\to (E_{\fS/\cD})\dual$ be the perfect relative obstruction theory constructed
in \cite{CL} by use of $(u_{\fS,i},\psi_{\fS})$.  (Its pullback to $\cY$ via $\cY\to\fS$ is the relative perfect obstruction theory of $\cY\to\cD$.)
Tracing through the construction of the obstruction theory $\phi_{\fS/\cD}$ in
\cite[Prop 2.5]{CL}, using that the universal section $(u_{\fS,i},\psi_{\fS})$ is equivariant under the
tautological $\CC\sta$-linearization
$(\bar\varphi_1,\bar\varphi_2)$, we conclude that the obstruction theory $\phi_{\fS/\cD}$ is $\CC\sta$-equivariant
with respect to the same $\CC\sta$-linearization.

We consider the composite
$$ \ti\jmath: \tcy\lra \cY\mapright{\jmath} \fS.
$$
of the tautological $\tcy\to\cY$ with the open embedding $\jmath$ in \eqref{j}.
Since the obstruction theory of $\tcy\to\ti\cD$ is the pullback of that of $\fS\to\cD$, we have canonical
isomorphism
\beq\label{BB} \ti\jmath\sta E_{\fS/\cD}\mapright{\cong} E_{\tcy/\ti\cD}.
\eeq
Secondly, since $\ti\jmath$ is a $\cD$-morphism and $\tcy$ is constructed from $\cY\to\cD$ via a base change
$\ti\cD\to\cD$. the $\CC\sta$-action on $\fS$ lifts to a
unique $\CC\sta$-action on
$\tcy$ that makes $\ti\jmath$ $\CC\sta$-equivariant. Using \eqref{BB}, we endow $E_{\tcy/\ti\cD}$ the
$\CC\sta$-linearlization induced based on $(\bar\varphi_1,\bar\varphi_2)$. This way,
since $\phi_{\fS/\cD}$ is $\CC\sta$-equivariant (via $(\bar\varphi_1,\bar\varphi_2)$), we conclude
that $\phi_{\tcy/\ti\cD}$ is $\CC\sta$-equivariant with the $\CC\sta$-linearlization just introduced.

Finally, it is direct to check that the introduced $\CC\sta$-action on $\tcy$ restricting to $\tcy\lgst$
is the $\CC\sta$-action $\fm_t$ constructed in \eqref{mt}; also, by \cite[Prop 2.5]{CL},
$E_{\fS/\cD}=R\bul\pi_{\fS\ast}(\sL_\fS^{\oplus 5}\oplus \sP_\fS)$, and
the introduced $\CC\sta$-linearization on $E_{\tcy/\ti\cD}\cong \ti\jmath\sta E_{\fS/\cD}$
restricting to $\tcy\lgst$ gives the linearization $(\Phi_{1,0}, \Phi_{2,1})$ of
$\ti V_1\oplus\ti V_2$, in the notation of \eqref{Phi12}.

Since $\phi_{\tcy/\ti\cD}$ is $\CC\sta$-equivariant, the intrinsic normal cone
$\bC_{\tcy/\ti\cD}\sub h^1/h^1(E_{\tcy/\ti\cD})$ is $\CC\sta$-equivariant; therefore
$C\lgst$ is $\CC\sta$-equivariant under the $\CC\sta$-action $(\Phi_{1,0}, \Phi_{2,1})$,
which is equivalent to say that $C\lgst$ is homogeneous of weight $(0,1)$.
This proves the Lemma.
\end{proof}

\section{Proof of Proposition \ref{prop2}}

We study the cone $C\lgst\sub \ti V$ in this section.
Let
$\Delta=\tcX\lpri\cap\tcX\lgst$, and recall
$$\ti\Delta=\tcy\times_{\tcX}\Delta =\tcY\lgst\times_{\tcX\lgst}\Delta=W\times_M \Delta
$$
is a line bundle over $\Delta$. We continue to denote by $\gamma: \tcy\lgst\to\tcX\lgst$ the projection.

\begin{prop}\label{C-gst}
We have 
$C\lgst \cap (0\oplus \ti V_2)|_{\tcy\lgst-\ti\Delta}=\tcy\lgst-\ti\Delta \sub \tcy\lgst$;
further, there is a sub-line bundle $\cF \sub V_{2}|_{\Delta}$ so that, denoting
$\ti \cF=\gamma\sta \cF\sub \ti V_{2}|_{\ti\Delta}$,
$$C\lgst\cap (0\oplus \ti V_2)|_{\ti\Delta} \sub \ti\cF
\sub \ti V_2|_{\ti\Delta}.
$$
\end{prop}

The proposition will be proved via a sequence of Lemmas.
First, following the argument in \cite[Sect 5.2]{CL}, the relative obstruction theories of the triple
$(\cpx,\ti\cX,\tfd)$ fit into a compatible diagram of distinguished triangles:
\beq\label{tri}
\begin{CD}
\gamma^\ast (E_{\ti\cX/\tfd})\dual @>>> (E_{\cpx/\ti\cD})\dual@>>> (E_{\cpx/\ti\cX})\dual @>{+1}>>\\
@VV{\gamma^\ast\phi_{\ti\cX/\tfd}}V @VV{\phi_{\cpx/\ti\cD}}V@VV{\phi_{\cpx/\tcX}}V\\
\gamma^\ast L^{\bullet}_{\ti\cX/\tfd} @>>> L^{\bullet}_{\cpx/\ti\cD}@>>>L^{\bullet}_{\cpx/\ti\cX}@>{+1}>>.
\end{CD}
\eeq
Here $E_{\tcy/\ti\cD}$ is given by \eqref{ty-o};
$E_{\ti\cX/\tfd}=R\bul\pi_{\tcX\ast} \sL_{\tcX}^{\oplus 5}$ and $E_{\tcy/\tcX}=R\bul\pi_{\tcy\ast} \sP_{\tcy}$. (As
usual, 
$\sL_{\tcy}=f_{\tcy}\sta\sO(1)$ and
$\sP_{\tcy}=\sL_{\tcy}^{\otimes(-5)}\otimes\omega_{\cC_{\tcy}/\tcy}$.)
Taking the cohomologies of the duals of the top row over
$\tcy\lgst$, we obtain  exact sequences of sheaves on $\tcy\lgst$: 
\beq\label{X-M}
\begin{CD}
H^1(E_{\cpx/\ti\cX}|_{\tcy\lgst}) @>>> H^1(E_{\cpx/\ti\cD}|_{\tcy\lgst})@>>>  H^0(\gamma^\ast E_{\ti\cX/\tfd}|_{\tcy\lgst})
\\
@|@|@|\\
\ti V_2@>{\ti\beta_1}>> \ti V=\ti V_1\oplus \ti V_2@>{\ti\beta_2}>>  \ti V_1.
\end {CD}
\eeq
Here the vertical identities are given by the explicit form of the complexes $E_{\tcy/\ti\cD}$, etc.;
$\ti\beta_1$ (resp. $\ti\beta_2$) is the tautological inclusion (resp. projection),
and the diagram commutes that follows from the proof of \cite[Prop 2.5 and 3.1]{CL}.

Let ${\cN}$  be the coarse moduli space of the the stack $h^1/h^0((L\bul_{\cpx/\ti\cX})\dual)|_{\ti\cY\lgst}$ relative to $\ti\cY\lgst$.
Then $\phi_{\cpx/\ti\cX}\dual$ induces an embedding
$${\cN}\sub H^1(E_{\ti\cY/\tcX})=\ti V_2.
$$

\begin{lemm}\label{lzo}
Viewing ${\cN}$ as a substack of $\ti V=\ti V_1\oplus \ti V_2$ via $\ti\beta_1$ in (\ref{X-M}), then
\beq\label{cone-int} \mathrm{Supp}(C\lgst)\cap(0\oplus \ti V_2)\sub {\cN}.
\eeq
\end{lemm}

\begin{proof}
It suffices to show that for any closed $y\in \ti\cY\lgst$, we have
$$  (C\lgst\cap(0\oplus \ti V_2))|_y\sub {\cN}_y\defeq {\cN}\times_{\ti\cY\lgst} y.
$$
First, it is shown in \cite{BF} that
${\cN}_y= H^1((L\bul_{\cpx/\ti\cX})\dual|_{y})$.
Dualizing \eqref{tri}, restricting it to $y$ and taking its cohomology groups,
we obtain the
commutative diagram of arrows
\beq\label{key}
\begin{CD}
{\cN}_{y}
@>>>H^1((L\bul_{\ti\cY/\tfd})\dual|_{y})
@>>> H^1(\gamma^\ast (L\bul_{\ti\cX/\tfd})\dual|_{y})\\
@VV{}V@VV{\varrho_y}V@VV{}V\\
 \ti V_2|_{y}@>{\ti\beta_1|_y}>>(\ti V_1\oplus \ti V_2)|_{y}@>{\ti\beta_2|_y}>>\ti V_1|_{y}\, ,
\end{CD}
\eeq
where the vertical arrows are $H^1$ of $\phi\dual_{\cpx/\ti\cX}|_{y}$, $\phi\dual_{\cpx/\tfd}|_{y}$
and $\gamma^\ast\phi\dual_{\ti\cX/\tfd}|_{y}$, respectively,
and the bottom line follows from \eqref{X-M}.
Since $\phi_{\tcy/\ti\cD}$, etc., are perfect obstruction theories,
the vertical arrows in (\ref{key}) are injective.

We now prove \eqref{cone-int}. First, the inclusion $\bcgst\sub \bC_{\ti\cY/\tfd}=h^1/h^0((L\bul_{\ti\cY/\tfd})\dual)$
induces an inclusion of their
respective fibers over $y$
$$\bcgst|_y\sub h^1/h^0((L\bul_{\ti\cY/\tfd})\dual)|_{y}= h^1/h^0((L\bul_{\ti\cY/\tfd})\dual|_{y}).
$$
(Here we use the base change property of $h^1/h^0$-construction; see \cite[Lemm 2.3]{CL1}.)
Taking coarse moduli, we have $ C\lgst|_y\sub H^1((L\bul_{\ti\cY/\tfd})\dual|_{y})$.
Hence
$$(C\lgst\cap(0\oplus \ti V_2))|_y=C\lgst|_y\cap (0\oplus  \ti V_2|_y) \sub \image(\varrho_y) 
\cap \ker(\ti\beta_2|_y).
$$
Because in (\ref{key}) all vertical arrows are injective and the squares are commutative,
$$\image(\varrho_y) 
\cap \ker(\ti\beta_2|_y) \sub 0\oplus {\cN}_{y}\sub 0\oplus \ti V_2|_y.
$$
This proves \eqref{cone-int}.
\end{proof}


\begin{lemm}\label{bL}
The substack ${\cN}\sub \ti V_2$  is reduced  and its support
is the union of the zero section $0_{\ti V_2}\sub \ti V_2$ with the total space of a line subbundle
$\ti \cF\sub \ti V_2|_{\ti\Delta}$ that is the pullback of a line subbundle
$\cF\sub V_2|_{{}\Delta}$.
\end{lemm}

\begin{proof}
Let $y\in W=\tcy\lgst$ be a closed point. We pick a chart $\iota: \ti Y\to\ti\cY$ of $y\in\ti\cY$ and its associated embedding
given in Proposition \ref{coordinate}. (We follow the notation introduced there.)
Then by \'etale base change
$$\iota^\ast h^1/h^0((L\bul_{\cpx/\ti\cX})\dual)|_{\ti Y\lgst}\cong h^1/h^0((L\bul_{\ti Y/\ti X})\dual)|_{\ti Y\lgst}.
$$

We denote $Z'_{\ti X}=\ti X\times_Z Z'\cong \ti X\times \Ao$. Lemma \ref{coordinate} induces a closed
$\ti X$-embedding $\ti Y\to Z'_{\ti X}$.
Let $I$ be the ideal of $\ti Y$ in $Z'_{\ti X}$. By definition,
\beq\label{L-del}
L_{\ti Y/\ti X}^{\bullet \geq -1}\cong [I/I^2\mapright{\delta} \Omega_{Z'_{\ti X}/\ti X}|_{\ti Y}],
\eeq
where $\delta(a)=d_{\ti X} a$. ($d_{\ti X}$ is the relative differential that annihilates $\sO_{\ti X}$.)
We denote $\mathrm S\bul(\cdot)$ the symmetric product, 
and denote
$$N \defeq \spec \mathrm S_{\sO_{\ti Y\lgst}}\bul(I/I^2\otimes_{\sO_{\ti Y}}\sO_{\ti Y\lgst}).
$$
Following \cite{BF}, since $h^1/h^0((L\bul_{\ti Y/\ti X})\dual)= h^1/h^0((L^{\bullet\geq -1}_{\ti Y/\ti X})\dual)$, we have
$$
h^1/h^0((L\bul_{\ti Y/\ti X})\dual)|_{\ti Y\lgst}=[(\spec {\mathrm S}\bul_{\sO_{\ti Y}}( I/I^2))/(T_{Z'_{\ti X}/\ti X}|_{\ti Y})]|_{\ti Y\lgst}=[N/(T_{Z'_{\ti X}/\ti X}|_{\ti Y\lgst})].
$$
Here the $T_{Z'_{\ti X}/\ti X}|_{\ti Y\lgst}$-action on $N$ is induced by the arrow $\delta$ in \eqref{L-del}.


By Lemma \ref{coordinate}, $I=(wt)\sub\sO_{Z'_{\ti X}}=\sO_{\ti X}[t]$. Using $\ti Y\lgst=(w=0)\cap Z'_{\ti X}$, we obtain
$\delta|_{\ti Y\lgst}(wt)=d_{\ti X}(wt)|_{\ti Y\lgst} =w\cdot dt|_{\ti Y\lgst}=0$;
hence the $T_{Z'_{\ti X}/\ti X}|_{\ti Y\lgst}$-action on $N$ is trivial,
and the coarse moduli of $h^1/h^0((L_{\ti Y/\ti X})\dual)|_{\ti Y\lgst}$ is $N$; namely
$$\cN\times_{\tcY\lgst} \ti Y\lgst= N.
$$

We determine $N$.
Define $\alpha:\sO_{Z'}\to I/I^2\otimes_{\sO_{\ti Y}}\sO_{\ti Y\lgst}$ to be the homomorphism via $\alpha(1)=wt\otimes 1$;
because $I=(wt)$, $\alpha$ is surjective.
We observe
$\alpha(w)=w^2t\otimes 1=wt\otimes w=0$, because $w=0\in \sO_{\ti Y\lgst}$; we also have
 $\alpha(z_i)=wz_it\otimes 1=0$, because $wz_i=0$ in $\sO_{\ti X}\subset I$.
A further check shows that $\ker  \alpha=(w,z_1,z_2,z_3,z_4)$; hence
$I/I^2\otimes_{\sO_{\ti Y}}\sO_{\ti Y\lgst}\cong \sO_{\ti Y\lD}$, where
$\ti Y\lD=\ti Y\times_\cpx{\ti\Delta}$. As $\spec \mathrm S_{\sO_{\ti Y\lgst}} \bul\!(\sO_{\ti Y\lD})$
is the union of ${\ti Y\lgst}$ with a line bundle $\ti F$ over $\ti Y\lD$ so that
$\ti Y\lgst\cap \ti F=\ti Y_\Delta$,
this proves that
$N={\ti Y\lgst}\cup \ti F$ 
with $\ti Y\lgst\cap \ti F=\ti Y_\Delta$.

As $\ti Y\to \tcy$ is an arbitrary \'etale chart,
it shows that there is a line bundle $\ti\cF$ on $\ti\Delta\sub \tcy\lgst$
such that
\begin{enumerate}
\item both $\tcy\lgst$ are $\ti\cF$ are closed substacks of $\cN$;
\item $\cN={\tcY\lgst}\cup\ti\cF$ and $\tcy\lgst\cap \ti\cF=\ti\Delta$,
where $\tcy\lgst\sub\cN$ is the zero section $0_{\barV_2}$.
\end{enumerate}
 Since $\cN\sub \ti V_2$ is a subcone,
$\ti\cF=\cN\times_{\tcy\lgst} \ti\Delta\sub \ti V_2|_{\ti\Delta}$ is also a subcone and we conclude that $\ti\cF\sub \ti V_2|_{\ti\Delta}$ is
a subline bundle.


Finally, we show that there is a subline line bundle $\cF\sub  V_2|_{\Delta}$ that
pullbacks to $\ti\cF\sub \ti V_2|_{\ti\Delta}$.
For this, we use the $\CC\sta$-action on $\tcy\lgst$ introduced in the proof of
Lemma \ref{5.7}. By Lemma \ref{5.7}, $\tcy\lgst-\tcX\lgst\to\tcX\lgst$ is a $\CC\sta$-quotient
morphism. By its construction, the relative obstruction theory of $\tcy\to\tcX$ 
is $\CC\sta$-equivariant. Thus $\ti\cF|_{\tcy\lgst-\tcX\lgst}\sub \ti V_2|_{\tcy\lgst-\tcX\lgst}$ is $\CC\sta$-equivariant,
where the $\CC\sta$-action on $\ti V_2|_{\tcy\lgst-\tcX\lgst}$ is via the linearization $\Phi_{2,1}$.

As this linearization $\Phi_{2,1}$ differs from $\Phi_{2,0}$ by a scalar multiplication, and because
$\ti\cF|_y\sub \ti V_2|_y$ for $y\in \ti\Delta$ is a linear
subspace, $\ti\cF|_{\ti\Delta}\sub \ti V_2|_{\ti\Delta}$ is also invariant under the linearization
$\Phi_{2,0}$. Since $\Phi_{2,0}$ is induced from the pullback $\ti V_2=\gamma\sta V_2$, by descent theory,
$\ti\cF|_{\ti\Delta-\Delta}\sub \ti V_2|_{\ti\Delta-\Delta}$ descends to a subline bundle $\cF\sub V_2|_\Delta$.
(As usual, we view $\Delta\sub\ti\Delta$ via the zero section of $W=\text{Total}(L)$.)
Thus $\ti\cF|_{\ti\Delta-\Delta}=\gamma\sta\cF|_{\ti\Delta-\Delta}\sub\ti V_2|_{\ti\Delta-\Delta}$.
Finally, since $\ti\Delta-\Delta$ is dense in $\ti\Delta$,  we have $\ti\cF=\gamma\sta\cF\sub \ti V_2|_{\ti\Delta}$.
This proves the Lemma.
\end{proof}

We now apply Corollary \ref{count} to the situation $W=\tcy\lgst$ and $M=\tcX\lgst$.
As will be clear later, we will work with each individual component $W_\mu=\tcy_\mu$ of $\tcy\lgst$,
and the corresponding components $M_\mu=\tcX_\mu$. Accordingly, we will add the
subscript $\mu$ to denote the corresponding objects restricting to $M_\mu$, $W_\mu$ or $\barW_\mu=\barW\times_M M_\mu$.
For instance, $\barV_{i,\mu}=\barV_i|_{\barW_\mu}$, $\Delta_\mu=\Delta\cap M_\mu$ and
$\ti\Delta_\mu=\ti\Delta\times_M M_\mu$, etc.

Following \eqref{split-0}, the cycle $[C\lgst]\in Z\lsta \ti V$ is a union
$[C\lgst]=\sum_{\mu\vdash d} \iota_{\mu\ast}[C_\mu]$,
where $[C_\mu]\in Z\lsta \ti V_\mu$ and $\iota_\mu: \ti V_\mu\to \ti V$ is the inclusion.
Here as argued before \eqref{C-dec}, $[C_\mu]=\sum_k n_{\mu,k} [C_{\mu,k}]$ could be a union of
integral multiples of irreducible components $C_{\mu,k}$.
Let $\barC_\mu$ be the closure of $C_\mu$ in $\barV_\mu$, which by the given irreducible decomposition
is $\sum_k n_{\mu,k} [\overline C_{\mu,k}]$; we let $\barC_{\mu,k,b}=\barC_{\mu,k}\cap(0\oplus \barV_{2,\mu})$,
and let
$$R_\mu=\sum_k n_{\mu,k} N_{\barC_{\mu,k,b}} \barC_{\mu,k}
$$
be the normal cone to $\barC_{\mu,b}=\sum_k n_{\mu,k} \overline C_{\mu,k,b}$ in $\barC_\mu$.

By Corollary \ref{count}, we have
\beq\label{Rmu}\deg 0^!_{\ti\xi,\loc}[C\lgst]=\deg 0^!_{\barV}[\barC\lgst]=\sum_{\mu\vdash d} \deg 0^!_{\barV_\mu}[R_\mu].
\eeq

We look more closely the term $0^!_{\barV_\mu}[R_\mu]$.  By Lemma \ref{bL} we have a line subbundle
$\cF_\mu:=\cF|_{\Delta_\mu}\black\sub V_{2,\mu}|_{\Delta_\mu}=V_2|_{\Delta_\mu}$. By Lemma \ref{compactify}, Lemma \ref{5.7} and Proposition \ref{C-gst},
$[\barC_{\mu,b}]$ lies over
the union of the zero section $0_{\barV_\mu}$ with the total space
$\overline F_\mu=\text{Total}(\overline\cF_\mu)$, 
where
$\overline{\Delta}_\mu=\barW_\mu\times_{M_\mu}\Delta_\mu$,
and $\bar\gamma_\mu:= \bar\gamma|_{\barW_\mu}:\barW_\mu\to M_\mu$. We let
\beq\label{Zmu}
Z_\mu=0_{\barV_{2,\mu}}\cup \overline F_\mu.
\eeq
Thus $R_\mu$ lies in the bundle $\barV_{1,\mu}\times_{\barW_\mu}Z_\mu$ over $Z_\mu$, where $Z_\mu\to \barW_\mu$ is
the identity map when restricted to $0_{\barV_{2,\mu}}$, and is composite of the projection $\overline F_\mu\to \overline\Delta_\mu
\to \barW_\mu$.

We claim that $\dim \barW_\mu=5d+4$, $\dim R_\mu=5d+6$,
$\rank \barV_{1,\mu}=5$ and $\rank \barV_{2,\mu}=5d+1$. Indeed, by Proposition \ref{coordinate}, we know
that all $\tcX_\mu$ are smooth and are of equal dimensions. Thus
$$\dim \tcX_\mu=\dim\tcX_{(d)}=\dim \barM_{0,1}(\Pf,d)+\dim \barM_{1,1}=5d+3.
$$
As $\tcy_\mu$ is a line bundle over $\tcX_\mu$, we obtain $\dim\barW_\mu=\dim \tcy_\mu=5d+4$.
For $\dim R_\mu$, since $C_\mu$ has pure-dimension, and $\dim R_\mu=\dim C_\mu$,
we only need to verify $\dim C_\mu=\dim \tcy_\mu+2$. But this follows from (2) of Lemma \ref{cone1}.
The remainder two identities follows from Riamenn-Roch theorem.

We denote by $|\overline C_{\mu,b}|$ the support of $\overline C_{\mu,b}$, which is the union
$\cup_k \overline C_{\mu,k,b}$.
Since $R_\mu$ is a normal cone, by its construction, it is a
cone inside the vector bundle $\barV_{1,\mu}\times_{\barW_\mu}|\barC_{\mu,b}|\sub \barV_{1,\mu}\oplus \barV_{2,\mu}$.
(Here as total space of vector bundles, $\barV_{1,\mu}\times_{\barW_\mu}\barV_{2,\mu}=\barV_{1,\mu}\oplus\barV_{2,\mu}$.)

Therefore, we have
\beq\label{int-11}
0_{\barV_{1,\mu}}^![R_\mu]\in A_{5d+1} |\barC_{\mu,b}|.
\eeq

We write $0_{\barV_{1,\mu}}^![R_\mu]=[P_{\mu, 1}]+[P_{\mu,2}]$, where $P_{\mu,1}\sub 0_{\barV_\mu}$
and $P_{\mu,2}\sub \overline F_{\mu}$. 

\begin{lemm}\label{6.4}
We have the vanishing $\deg 0_{\barV_{2,\mu}}^![P_{\mu,i}]=0$ in the following cases:
\begin{enumerate}
\item $\mu\vdash d$ and $i=2$;
\item $\mu\ne(d)\vdash d$ and $i=1$.
\end{enumerate}
\end{lemm}

\begin{proof}
We prove the vanishing for $i=2$. By Lemma \ref{bL}, the subline bundle $\cF_\mu\sub V_1|_{\Delta_\mu}$ pullbacks to $\overline\cF_\mu\sub \barV_{2,\mu}|_{\overline\Delta_\mu}$.
Let $\eta_\mu: \overline\cF_\mu\to \cF_\mu$ be the projection; it is proper since $\overline\Delta\to\Delta$ is proper. Then by the projection formula,
$$\deg 0^!_{\barV_{2,\mu}}[P_{\mu,2}]=\deg 0^!_{V_{2,\mu}}(\eta_{\mu\ast}[P_{\mu,2}]).
$$
Since $\eta_{\mu\ast}[P_{\mu,2}]\in A_{5d+1} \cF_\mu$ and $\dim \cF_\mu=\dim \Delta_\mu+1=(\dim \tcX\lpri-1)+1=5d$,
we have $\eta_{\mu\ast}[P_{\mu,2}]=0$. This proves the first vanishing.

To prove the second vanishing, we will construct a proper DM stack $\ti \cB_\mu$, a vector bundle $\ti\sK_\mu$ on $\ti\cB_\mu$,
and a proper morphism with an isomorphism
\beq\label{rhomu}\rho_\mu: M_\mu=\tcX_\mu\lra \ti\cB_\mu\and \rho_\mu\sta \ti\sK_\mu\dual\cong V_{2,\mu},
\eeq
such that
$\dim \ti\cB_{(d)}=5d+1$, and $\dim \ti\cB_\mu\le 5d$ for $\mu\ne (d)$. Once $(\ti\cB_\mu, \ti\sK_\mu)$ is
constructed, we let $\delta_\mu: \barV_{2,\mu}\to \ti\sK_\mu\dual$ be the projection
induced by the mentioned isomorphism; then, for $\mu\ne (d)$, by the projection
formula we have $\deg 0^!_{\barV_{2,\mu}}[P_{\mu,1}]=\deg 0^!_{\sK_\mu\dual}(\delta_{\mu\ast}[P_{\mu,1}])=0$
because $P_{\mu,1}\sub 0_{\barV_{2,\mu}}$ and $A_{5d+1} \ti\cB_\mu=0$ for dimension reason.

Constructing $\ti\cB_\mu$ and $\ti\sK_{\mu}$ with the required properties will occupy the remainder of this Section.
\end{proof}

\def\lle{_{\tcX_\mu,\mathrm{pr}}}
\def\llr{_{\tcX_\mu,\mathrm{tl}}}
\def\llg{_{\mu,\mathrm{gst}}}
\def\fh{{\mathfrak h}}
\def\bv{{\mathbf v}}
\def\sP{{\mathscr P}}
\def\Moo{\bcM_{1,1}}
\def\sP{{\mathscr P}}
\def\sQ{{\mathscr Q}}

We first state a decomposition result, which follows from the construction in \cite{HL1}. Let $\mu=(d_1,\cdots,d_\ell)$ be
a partition of $d$ and let $(f_{\ti\cX\lmu},\pi_{\ti\cX\lmu},\cC_{\ti\cX\lmu})$ be the tautological family of $\ti\cX\lmu$.
By the construction of $\tcX\lgst$, the map associated to a closed point in $\tcX\lmu$ is by attaching $\ell$ one-pointed
$[u_i,C_i,p_i]\in \bcM_{0,1}(\Pf,d_i)$ such that $u_i(p_i)=u_j(p_j)$ for all $i$, $j$ to an $\ell$-pointed
stable elliptic curve. We state this in the family version

\begin{prop}\lab{lem6.1}
The tautological family $\cC_{\tcX_\mu}\to\tcX_\mu$
admits an $\ell$-section $\Sigma_\mu\sub \cC_{\tcX_\mu}$ (a codimension one closed substack, proper, and an $\ell$-\'etale
cover of $\tcX_\mu$) 
that lies in the locus of nodal points of the fibers of $\cC_{\tcX_\mu}/\tcX_\mu$, and splits $\cC_{\tcX_\mu}$ into
two families of curves: $\cC\lle$ and $\cC\llr$ ($\sub \cC_{\tcX_\mu}$), such that
\begin{enumerate}
\item
the pair $(\cC\lle,\Sigma_\mu)$ is a family of $\ell$-pointed (unordered)
stable genus one curves; the
morphism $f_{\tcX_\mu}$ is constant along fibers of $\cC\lle\to \tcX_\mu$;
\item the pair $(\cC\llr,\Sigma_\mu)$
is a family of $\ell$-pointed (unordered) nodal rational curves over $\tcX_\mu$ such that each closed fiber of
$\cC\llr\to \tcX_\mu$ has $\ell$ connected components, and each connected component contains one marked point.
\end{enumerate}
\end{prop}

Here the subscript
``pr'' stands for the ``principal part'' and the subscript ``tl'' stands for the ``tail''.
We comment that the total space $\overline{\cC}\llr$ may have less than $\ell$ connected components.

\begin{proof}
The proof follows from the modular construction of $\tcX\lmu$ in \cite{HL1}.
\end{proof}

We now fix a partition $\mu$ and assume $\ell\ge 2$. Using this decomposition, we can relate $\tcX_\mu$ to a stack that
parameterizes the tails of $[u,C]\in\tcX_\mu$. 
We take the moduli of genus zero stable morphisms $\bcM_{0,1}(\Pf,d_i)$, considered as a stack over $\Pf$ via
the evaluation morphism $\ev_i: \bcM_{0,1}(\Pf,d_i)\to \Pf$ (of the marked points); we form
$$\cB_\mu=\bcM_{0,1}(\Pf,d_1) \times_\Pf\cdots\times_\Pf \bcM_{0,1}(\Pf,d_\ell).
$$
We let $S_\mu$ be the subgroup of permutations $\alpha\in S_\ell$ that leave
the $\ell$-tuple $(d_1,\cdots,d_\ell)$ invariant (i.e. $d_i=d_{\alpha(i)}$ for all $1\le i\le \ell$).
Each $\alpha\in S_\mu$ acts as an automorphism of $\cB_\mu$ by permuting the $i$-th and the $\alpha(i)$-th
factors of $\cB_\mu$. This gives an $S_\mu$-action on $\cB_\mu$; we define the stacky quotient
$$\ti\cB_\mu=\cB_\mu/S_\mu.
$$
Since $\bcM_{0,1}(\Pf,d_i)$ are proper DM stacks and have dimensions $5d_i+2$, $\ti\cB$ is a proper
DM-stack and
\beq\label{dim-B}
\dim \ti\cB_\mu=(5d+2l)-(4l-4)=5d-2l+4.
\eeq

For each $i$ we denote the universal family of $\barM_{0,1}(\Pf,d_i)$
by $(\pi_i,f_i):\cC_i\to \barM_{0,1}(\Pf,d_i)\times \Pf$ with
$s_i:\barM_{0,1}(\Pf,d_i)\to\cC_i$ the section of marked point.
We introduce 
$$\sK_i=\pi_{i\ast} f_i\sta \sO(5),
$$
a rank $5d_i+1$ locally free sheaf (vector bundle) on $ \barM_{0,1}(\Pf,d_i)$
with an evaluation homomoprhism $e_i: \sK_i\to \ev_i\sta \sO(5)$.

 We let
$v_i:\cB_\mu\to 
\bcM_{0,1}(\Pf,d_i)$
be the projection. Since $\cB_\mu$ is constructed as the fiber-product using the evaluations $\ev_i$, the
collection $\{\ev_i\}_{i=1}^\ell$ descends to a single evaluation morphism
$\ev: \cB_\mu\to\Pf$. We form a sheaf $\sK_\mu$ on $\cB_\mu$ defined via the exact sequence
$$0\lra \sK_\mu\lra \oplus_{i=1}^\ell v_i\sta\sK_i\mapright{\beta} \oplus_{j=1}^{\ell-1} \ev\sta \sO(5)\lra 0,
$$
where $\beta=(v_1\sta e_1-v_2\sta e_2,\cdots,v_{\ell-1}\sta  e_{\ell-1}-v_\ell\sta e_\ell)$.
Since each $e_i:\sK_i\to \ev_i\sta\sO(5)$ is surjective, $\sK_\mu$ is locally free.
Further, $\sK_\mu$ is invariant under $S_\mu$, thus descends to a locally
free sheaf $\ti\sK_\mu$ on $\ti\cB_\mu$. (Caution, the arrow $\beta$ is not invariant under $\sim$, due to the choice of
indexing.) As $\rank \sK_i=5d_i+1$, $\rank \ti\sK_\mu=5d+1$.

%
%
%
%

We define the desired morphism
$\rho\lmu:\tcX_\mu\to \ti\cB_\mu$ stated in \eqref{rhomu}.
Given any closed $x\in\tcX\lmu$, we let $[f_x, \cC_x]$ with $\Sigma_x\sub \cC_x$ be the
restriction of $f_{\tcX\lmu}$ and $\Sigma_\mu$ to $\cC_x=\cC_{\tcX\lmu}\times_{\tcX\lmu} x$. By Proposition \ref{lem6.1},
$\Sigma_x$ divides $\cC_x$ into a union of a genus one curve with $\ell$ one-pointed genus zero curves,
and the restriction of $f_x$ to these $\ell$ genus zero curves form $\ell$ one-pointed genus zero stable maps to $\Pf$, of
degrees $d_1,\cdots,d_\ell$, respectively. We label these $\ell$-stable maps as $[u_i,C_i,p_i]$
so that $[u_i]\in\barM_{0,1}(\Pf, d_i)$. We define
$\rho_\mu(x)$ be the equivalence class in $\ti\cB_\mu$ of $([u_1],\cdots,[u_\ell])\in\cB_\mu$.
It is routine to check that this pointwise definition defines a morphism $\rho_\mu$ as indicated.

We verify that $\rho_\mu\sta\ti \sK_\mu\dual\cong V_{2,\mu}$. First, by Serre duality,
$$V_{2,\mu}=R^1\pi_{\tcX\lmu\ast}\sP_{\ti\cX\lmu}\cong (\pi_{\tcX\lmu \ast} f^\ast_{\tcX\lmu}\sO(5))\dual.
$$
Let $f\llr$ and $\pi\llr$ be the restrictions of $f_{\tcX\lmu}$ and $\pi_{\tcX\lmu}$ to $\cC\llr$.
We obtain the restriction homomorphism
\beq\label{restrict}
\pi_{\tcX\lmu \ast} f_{\tcX\lmu}\sO(5)\mapright{\sub} (\pi\llr)\lsta f\llr\sta\sO(5).
\eeq
Since fibers of $\cC\lle\to \tcX\lmu$ are connected, and $f_{\tcX\lmu}$ restricting to
them are constants, (\ref{restrict}) is injective and its
cokernel is the difference of the evaluations along $\Sigma_\mu\sub\cC\llr$.
On the other hand, denoting $(\oplus_{i=1}^\ell v_i\sta\sK_i)/S_\mu$ the quotient of
$\oplus_{i=1}^\ell v_i\sta\sK_i$ over $\cB_\mu$ by $S_\mu$, (which is its descent to $\ti\cB_\mu$,) then
via $\rho_\mu$, we have a canonical isomorphism
$$\rho_\mu\sta((\oplus_{i=1}^\ell v_i\sta\sK_i)/S_\mu)\cong (\pi\llr)\lsta f\llr\sta\sO(5).
$$
A direct inspection shows that under this isomorphism, $\rho_\mu\sta\ti\sK_\mu\cong \pi_{\tcX\lmu \ast} f^\ast_{\tcX\lmu}\sO(5)$.
This proves $\rho_\mu\sta\ti\sK_\mu\dual \cong V_{2,\mu}$.

In case $\mu=(d)$, we let
$\ti\cB_{(d)}=\barM_0(\Pf,d)$, and all others are the same.

\begin{proof}[Completing the proof of Lemma \ref{6.4}]
The pair $(\ti\cB_\mu,\sK_\mu)$ satisfies the requirements stated in the proof of Lemma \ref{6.4}, except possibly the
dimensions part. By \eqref{dim-B}, for $\ell\ge 2$, we have $\dim\ti\cB_\mu\le 5d$;
for $\mu=(d)$, we have $\dim \ti\cB_{(d)}=\dim\barM_0(\Pf,d)=5d+1$, which are as required.
This completes the proof of Lemma \ref{6.4}.
\end{proof}

\section{proof of Proposition \ref{prop3}}
\def\dcirc{_\circ}
\def\ld{\dcirc}
\def\Wc{W\dcirc}
\def\barS{\overline S}

In this section, we treat the remainder case left out in Lemma \ref{6.4}.
Following the discussion in the previous section, we know that
$[P_{(d),1}]\in A_{5d+1} \barW_{\!(d)}$ and $\ti\cB_{(d)}$ is irreducible and of dimension $5d+1$;
thus 
\beq\label{c0}
(\rho_{(d)}\circ \bar\gamma_{(d)})_{\ast}[P_{(d),1}]=c[\ti\cB_{(d)}]
\eeq
for a $c\in\QQ$. Applying
the projection formula to $\rho_{(d)}\circ \bar\gamma_{(d)}$, we obtain
\beq\label{c?}\deg 0^!_{\barV_{2,(d)}}[P_{(d),1}]=c\cdot \deg c_{5d+1}(\ti\sK_{(d)})[\ti\cB_{(d)}]
=c\cdot (-1)^{5d+1}N_0(d)_Q.
\eeq
Therefore, all we need  is to determine $c$.

To this end, we need to determine the cone $C_{(d)}\sub \ti V|_{W_{(d)}}$.
By the construction in \cite{HL1}, $\barM_1(\Pf,d)$ contains an open substack consisting of stable maps $[u,C]$
such that their domains $C$ are union of elliptic curves with $\Po$, and the maps $u$ restricted to the elliptic curves are
constant and restricted to $\Po$ are regular embeddings $\Po\to\Pf$. We denote this open substack by
$M\dcirc\sub \barM_1(\Pf,d)$. By \cite{HL1}, $M\dcirc$ is away
from the blowing up loci of the modular blowing up $\tcX\to\cX=\barM_1(\Pf,d)$. Thus, the preimage
of $M\dcirc$ in $\tcX$ is identical to $M\dcirc$, and thus $\tcy\times_{\cX} M\dcirc=\cY\times_{\cX}M\dcirc$.

We denote
$W\dcirc=\cY\times_{\cX}M\dcirc$. By the prior discussion, we know $W\dcirc=\tcy\times_{\cX}M\dcirc\sub W_{(d)}$,
and is dense in $W_{(d)}$.
Since we need to work with various obstruction theories that are originally constructed for $\cY$ and $\cX$,
we will view $M\dcirc\sub M_{(d)}$ and also an open substack of $\cX$; same for $W\dcirc\sub W_{(d)}$ and
$W\dcirc\sub \cY$.

Since $W\dcirc\sub \cY$ is open, the obstruction theory $\phi_{\cY/\cD}$ restricted to $W\dcirc$ gives
a perfect relative obstruction theory $\phi_{\cY/\cD}|_{W\dcirc}: (E_{\cY/\cD})\dual|_{W\dcirc}\to L\bul_{W\dcirc/\cD}$.

\begin{lemm}\label{7.1}
Restricting to $W\dcirc$, we have
$$C_{(d)}|_{W\dcirc}=H^1((L\bul_{W\dcirc/\cD})\dual)\mapright{\sub} H^1(E_{\cY/\cD}|_{W\dcirc})=\ti V|_{W\dcirc}.
$$
Further, $H^1((L\bul_{W\dcirc/\cD})\dual)$ is a rank two locally free sheaf of $\sO_{W\dcirc}$-modules, and
the arrow above is injective with locally free cokernel.
\end{lemm}

\begin{proof}
Let $q\dcirc: W\dcirc\to\cD$ be the projection induced by $\cY\to\cD$;
let $\cD\ld=q\dcirc(W\dcirc)$ be its image stack. By the description of $W\dcirc$ and the
argument in the proof of Lemma \ref{cone1}, $\cD\ld{\sub} \cD$ is a smooth codimension two
locally closed substack. \black By \cite[Chap. III Prop. 3.2.4\black]{Illusie}, $H^1(L_{\cD\ld/\cD}\dual)$  
is isomorphic to the normal sheaf
to $\cD\ld$ in $\cD$, which is a rank two locally free sheaf on $\cD\ld$.

On the other hand, since $W\dcirc\to \cD\ld$ is smooth, \cite[ Chap. III Prop. 3.1.2 \black]{Illusie} 
 implies
that $H^i((L\bul_{W\dcirc/\cD\ld})\dual)=0$ for $i\geq 1$. Taking $H^1$ of the distinguished triangle
$$(L\bul_{W\dcirc/\cD\ld})\dual\lra (L\bul_{W\dcirc/\cD})\dual\lra q\dcirc^\ast(L\bul_{\cD\ld/\cD})\dual\mapright{+1},
$$
we obtain canonical
$$q\dcirc^{\ast} H^1((L\bul_{\cD\ld/\cD})\dual)\cong H^1((L\bul_{W\dcirc/\cD})\dual), 
$$
thus they are rank two locally free sheaves on $W\dcirc$.

Next, for any closed point $y\in \cY$, since $\phi_{\cY/\cD}$ is a perfect obstruction theory,
$H^1(\phi_{\cY/\cD}\dual|_y)$ is injective. This combined with $W\dcirc$ being smooth shows that
the arrow in the statement of Lemma \ref{7.1} is injective with locally free cokernel.

Finally, we show that $C_{(d)}|_{W\dcirc}=H^1((L\bul_{W\dcirc/\cD})\dual)$.
Because $W\dcirc\to \cD\ld$ is smooth and $\cD\ld\sub \cD$ is smooth of codimension two, we conclude that
$L\bul_{W\dcirc/\cD}$ is perfect of amplitude $[0,1]$, and
$H^i((L\bul_{W\dcirc/\cD})\dual)$ are locally free. Hence
the coarse moduli of $h^1/h^0((L\bul_{W\dcirc/\cD})\dual)$ is $H^1((L\bul_{W\dcirc/\cD})\dual)$.
By \cite[Prop 3.1.2]{BF}, $\bC_{\cY/\cD}|_{W\dcirc}=h^1/h^0((L\bul_{W\dcirc/\cD})\dual)$. Taking coarse moduli, we obtain
$C_{(d)}|_{W\dcirc}= H^1(L\bul_{W\dcirc/\cD})\dual)$.
This proves the Lemma.
\end{proof}

Let $\gamma_\circ: W\dcirc\to M\dcirc$ be the projection induced by $\cY\to\cX$.
We denote by $\cM$ the Artin stack of genus one nodal curves; (this is consistent with
$\cfM$ is the stack of weighted genus one nodal curves). Since $\cfM\to\cM$ is \'etale, the obstruction
theory of $\cX\to \cM$ is the same as that of $\cX\to\cfM$.

We determine the subsheaf $H^1((L\bul_{W_\circ/\cD})\dual)\sub \ti V|_{W\dcirc}$ by
studying the following diagrams:
\beq\label{diag-3}
\begin{CD}
H^1((L\bul_{W\dcirc/\cD})\dual) @>\alpha_1>> \gamma_\circ\sta H^1((L\bul_{M\dcirc/\cD})\dual)
@>\alpha_2>> \gamma_\circ\sta H^1((L\bul_{M\dcirc/\cM})\dual)\\
@VV{H^1(\phi_{\cY/\cD}\dual)}V @VV{H^1(\phi_{\cX/\cD}\dual)}V @VV{H^1(\phi_{\cX/\cM}\dual)}V\\
H^1(E_{\cY/\cD}|_{W\dcirc})@>{\ti\beta_2|_{W_\circ}}>> \gamma_\circ\sta H^1(E_{\cX/\cD}|_{M\dcirc})
@>\gamma_\circ\sta\beta_\circ>>
\gamma_\circ\sta H^1(E_{\cX/\cM}|_{M\dcirc}).
\end{CD}
\eeq
Here, $\ti\beta_i$ is defined in \eqref{X-M} and $\beta_\circ$ is the tautological projection
induced by the comparison of the obstruction theories of $\cX\to \cD$ and $\cX\to\cM$:
\beq\label{beta2}\beta_\circ: H^1(E_{\cX/\cD}|_{M\dcirc}) \lra  H^1(E_{\cX/\cM}|_{M\dcirc}).
\eeq

We comment that the left commutative square is induced by that the obstruction theories (relative to $\cD$)
of $\cY\sub C(\pi\lsta(\sL^{\oplus 5}\oplus \sP))$ and of $\cX\sub C(\pi\lsta(\sL^{\oplus 5}))$
are compatible under
$C(\pi\lsta(\sL^{\oplus 5}\oplus \sP))\to C(\pi\lsta(\sL^{\oplus 5}))$ induced by the projection
$\pr:\sL^{\oplus 5}\oplus \sP\to \sL^{\oplus 5}$. The right square is commutative following
\cite[Lemm 2.8]{CL}.

We let $(f_{M_\circ}, \pi_{M_\circ}):\cC_{M_\circ}\to \Pf\times M\dcirc$ be the universal family of $M\dcirc\sub \cX$.

\begin{lemm}\label{7.2}
All sheaves in the diagram \eqref{diag-3}are locally free sheaves of $\sO_{W\dcirc}$-modules;
all vertical arrows are injective with locally free cokernels; the arrow $\alpha_1$ is an isomorphism;
the arrow $\alpha_2$ is surjective and has rank one kernel; the arrow $\ti\beta_2|_{W_\circ}$ is the obvious
projection $\ti V|_{W\dcirc}=(\ti V_1\oplus \ti V_2)|_{W\dcirc}\to \ti V_1|_{W\dcirc}$;
the arrow $\beta_\circ$ is the projection
$$\beta_\circ: H^1(E_{\cX/\cD}|_{M\dcirc}) = R^1\pi_{M_\circ\ast} f_{M_\circ}\sta \sO(1)^{\oplus 5}\lra
H^1(E_{\cX/\cM}|_{M\dcirc})=R^1\pi_{M_\circ\ast} f_{M_\circ}\sta T_{\Pf}
$$
induced by the tautological projection $\sO(1)^{\oplus 5}\to T_{\Pf}$ in the Euler sequence of $\Pf$.
Finally, the cokernels of the middle and the third vertical arrows are isomorphic.
\end{lemm}

\begin{proof}
We let $\cM\ld\sub\cM$ be the image stack of $M\dcirc\to\cM$. By the description of
$M\dcirc$, $\cM\ld$ is a locally closed smooth divisor in $\cM$. Thus the normal sheaf
$\cN_{\cM\ld/\cM}$ is invertible and the canonical
\beq\label{pro-N}
\cN_{\cD\ld/\cD}\lra \cN_{\cM\ld/\cM}\otimes_{\sO_{\cM\ld}}\!\!\sO_{\cD\ld}
\eeq
is surjective.

Following the proof of Lemma \ref{7.1},
$H^1((L\bul_{M\dcirc/\cD})\dual)$ and $H^1((L\bul_{M\dcirc/\cM})\dual)$ are canonically isomorphic
to the pullbacks of the normal sheaves $\cN_{\cD\ld/\cD}$
and $\cN_{\cM\ld/\cM}$, respectively, and the arrow $\alpha_1$ and $\alpha_2$ are induced by
the identity map of $\cN_{\cD\ld/\cD}$ and \eqref{pro-N}, respectively. This proves the statements about
the sheaves and arrows in the top horizontal  line.

Parallel to the proof Lemma \ref{7.1}, we obtain that the vertical arrows are injective with locally free cokernels.

By definition of $\ti V_i$, the first two sheaves in the lower horizontal line are $\ti V|_{W\dcirc}$ and
$\gamma_\circ\sta V_1|_{M\dcirc}=\ti V_1|_{W\dcirc}$, and the arrow $\ti\beta_2|_{W_\circ}$ is the obvious projection
as stated. The statement about $\gamma_\circ\sta \beta_\circ$ and $\beta_\circ$ follows fom \cite[Lemm 2.8]{CL}.

A direct calculation shows that $R^1\pi_{M_\circ\ast} f_{M_\circ}\sta \sO(1)^{\oplus 5}$ and
$R^1\pi_{M_\circ\ast} f_{M_\circ}\sta T_{\Pf}$ have rank five and four, respectively, and $\gamma_\circ\sta\beta_\circ$ is surjective,
therefore $\ker(\gamma_\circ\sta\beta_\circ)$ is an invertible sheaf, and is isomorphic to $\ker(\alpha_2)$,
using that the middle and the third vertical arrows are injective with locally free cokernel. Consequently, the
cokernels of the middle and the third vertical arrows are isomorphic.
These complete the proof of the Lemma.
\end{proof}

We now determine the image sheaf of the third vertical arrow in \eqref{diag-3}.
Let $\xi=[u,C]\in M\dcirc$ be a closed point. By the description of $M\dcirc$, $C=E\cup R$ such that
$E$ (resp. $R$) is a nodal genus one curve (resp. $R\cong \Po$) and $p=E\cap R$ is a node of $C$,
and that $u|_E=\text{const}.$ and $u|_R:R\to\Pf$ is a regular embedding.
Let $\underline{\xi}\in \cM$ be the image of $\xi$ under the tautological $M\dcirc\to\cM$;
let $v\in T_{\cM, \underline{\xi}}$ be a non-zero vector normal to $\cM_{(d)}\sub \cM$.
According to the obstruction theory, the image of $H^1(\phi_{\cX/\cM}\dual)|_\xi$ in
$H^1(E_{\cX/\cM})|_\xi$ is the linear span of the image of $v$ under the composite
\beq\label{Ob-1} \ob_\Pf: H^0((L\bul_\cM)\dual|_{\underline{\xi}})\lra H^1((L\bul_{M\dcirc/\cM})\dual|_\xi)\lra
H^1(E_{\cX/\cM}|_\xi)
\eeq
induced by the obstruction theory $\phi_{\cX/\cM}$ (it is the obstruction assignment map).
Because $u|_E$ is constant, we have
$$H^1(E_{\cX/\cM}|_\xi)=H^1(u\sta T_\Pf)=H^1(\sO_E)\otimes_{\CC} (u\sta T_{\Pf})|_p=H^1(\sO_E)\otimes_\CC T_{\Pf,u(p)}.
$$

\begin{lemm}\label{7.3}
The linear span of the image of $v\in H^0((L\bul_\cM)\dual|_\xi)$ in $H^1(E_{\cX/\cM}|_\xi)$ is the subspace
$$H^1(\sO_E)\otimes_{\CC} u\lsta(T_{R,p})\sub H^1(\sO_E)\otimes_\CC T_{\Pf,u(p)}.
$$
\end{lemm}

\begin{proof}
Let $H=u(R)\sub \Pf$. Since $u|_R$ is a regular embedding, $H\sub \Pf$ is a smooth curve. We let $u': C\to H$
be the factorization of $u:C\to\Pf$. Thus $\xi'=[u',C]\in \barM_1(H, d')$, where $d'=u'\lsta[R]\in H_2(H,\ZZ)$.
We let $\barM_1(H,d')\to\cM$ be the tautological projection; thus $\underline{\xi}\in\cM$ is also the image of $\xi'$.

Since $H\cong \Po$, it is known that there is no first order deformation of $[u',C]$ in $\barM_1(H,d')$ whose image in
$T_{\underline{\xi}}\cM$ is $v$ (cf. \cite{Z}, \cite{HL1}).
Thus the image of $v$ under the obstruction assignment
$$\ob_H: H^0((L\bul_\cM)\dual)|_{\underline{\xi}})\lra 
H^1(E_{\barM_1(H,d')/\cM}|_{\xi'})=H^1(\sO_E)\otimes_{\CC} T_{H, u'(p)}
$$
is non-zero. Since $\dim H^1(\sO_E)=\dim T_{H,u'(p)}=1$, the linear span of $\ob_H(v)$ is
$H^1(\sO_E)\otimes_{\CC} T_{H, u'(p)}$.

Then, because the obstruction theories of moduli of stable morphisms to schemes are compatible via morphisms
between schemes, we conclude that the linear span of the image $\ob_\Pf(v)\sub H^1(E_{\cX/\cM}|_\xi)$
is identical to the image of the linear span of $\ob_H(u')$ under the canonical
$$
H^1(E_{\barM_1(H,d')/\cM}|_{\xi'})\lra H^1(E_{\barM_1(\Pf,d)/\cM}|_{\xi})= H^1(E_{\cX/\cM}|_\xi). 
$$
Adding $T_{H, u'(p)}=u\lsta(T_{R,p})$ as subspace in $T_{\Pf,u(p)}$, we prove the Lemma.
\end{proof}

We consider
$H^1((L\bul_{M\dcirc/\cD})\dual)\to H^1(E_{\cX/\cD}|_{M\dcirc})=V_1|_{M\dcirc}$ of the middle
vertical arrow in \eqref{diag-3}. By Lemma \ref{7.2}, $H^1((L\bul_{M\dcirc/\cD})\dual)$ is a rank two locally free sheaf
on $M\dcirc$, and the arrow is injective with locally free cokernel.  Let $S_\circ\sub V_1|_{M_\circ}$ be
associated sub-vector bundle.

We continue to denote by $\bar\gamma: \barW\to M$ the projection (cf. Section five). We let $\barW\dcirc=\barW\times_M M\dcirc$ and
$\bar\gamma\dcirc: \barW\dcirc\to M\dcirc$ the projection.
Recall that $\barV_1=\bar\gamma\sta V_1(-D_\infty)$. Thus $S_\circ\sub V_1|_{M\dcirc}$ provides a subbundle
\beq\label{sub-S}
\barS_\circ=\bar\gamma_\circ\sta S_\circ(-D_\infty)\sub \barV_1|_{\barW_\circ}.
\eeq

We let
$$\eta_\circ:\barV_1|_{\barW_\circ}\to \barV|_{\barW_\circ}=(\barV_1\oplus \barV_2)|_{\barW_\circ}
$$
be the inclusion;
let ${j}_\circ: \barW_\circ\to \barW$ be the open embedding, which is flat.
Recall $R_{(d)}=N_{\barC_{(d),b}}\overline C_{(d)}$ (cf. before \eqref{Rmu}, see also \eqref{int-11}).

\begin{lemm}\label{pi1}
As cycles, we have
\beq\label{NB}
{j}_{\circ}\sta [R_{(d)}]=\eta_{\circ\ast}[\barS_\circ]\in Z\lsta (\barV_1|_{\barW_\circ}).
\eeq
%
\end{lemm}

\begin{proof}
Lemma \ref{7.1} shows that $C_{(d)}\times_{W}{\Wc}$ is a rank two subbundle of $\ti V|_{\Wc}$.
By Proposition \ref{C-gst} and  Lemma \ref{compactify}, we have $\barC_{(d),b}\sub \barV_2$,
$\barC_{(d),b}\cap  \barV_2|_{\oWc}=
0_{\bar2}\times_{\barW} \barW_\circ$,
and $\barC_{(d)}\times_{\barW}\barW_\circ$ is a rank two subbundle of $\barV|_{\oWc}$.
Further, they fit into the commutative diagram (the left one is a Cartesian product)
$$
\begin{CD}
0_{\barV}\times_{\barW}\barW_\circ @>>>\barC_{(d)}\times_{\barW}{\oWc} @>\cong>>\barS_\circ \\
@VV{\subseteq}V@VV{\subseteq}V@VV{\subseteq}V \\
\barV_{2}|_{\oWc}@>{\bar\beta_1|_{\barW_\circ}}>> \barV|_{\barW_\circ}=(\barV_{1}\oplus \barV_{2})|_{\oWc}
@>{\bar\beta_2|_{\barW_\circ}}>>  \barV_1|_{\barW_\circ},
\end{CD}
$$
where ${\bar\beta_1|_{\barW_\circ}}$ and ${\bar\beta_2|_{\barW_\circ}}$ are the obvious inclusion and projection,
which implies that
$$R_{(d)}\times_{\barW}\barW_\circ= (N_{\barC_{(d),b}}\overline C_{(d)})\times_{\barW} {\oWc}=\eta_\circ(\barS_\circ) \sub \barV|_{\oWc}.
$$
Since $j_\circ:\barW_\circ\to \barW$ is an open embedding,
$${j}_\circ\sta[R_{(d)}]=[R_{(d)}\times_{\barW}\barW_\circ]=[\eta_\circ(\barS_\circ)]=
\eta_{\circ\ast}[\barS_\circ]\in Z\lsta(\barV_1|_{\barW_\circ}).
$$
This proves the Lemma.
\end{proof}

We pick a degree $d$ regular embedding $h:\Po\to\Pf$, viewed as a closed point in $\ti\cB_{(d)}=\barM_0(\Pf,d)$. We form
$$M_h=\{[u,C]\in M_\circ \mid u|_R\cong h\}\sub M_\circ.
$$
Using the convention introduced in the proof of Lemma \ref{6.4}, we have that $\barW_{(d)}\cup\bar F_{(d)}\sub \barV_{2,(d)}$,
where $\barW_{(d)}$ is the zero-section of $\barV_{2,(d)}$. We form
$\barW_h=\barW_\circ\times_{M_\circ} M_h$ and the inclusions 
$$j_h: \barW_h\lra  \barW_{(d)}\cup \bar F_{(d)},
\and J_h: \barV_1|_{\barW_h}\lra \barV_1\times_{\barW_{(d)}}(\barW_{(d)}\cup \bar F_{(d)}),
$$
where the last term is viewed as a vector bundle over $\barW_{(d)}\cup \bar F_{(d)}$.
Since $j_h(\barW_h)\cap\bar F_{(d)}=\emptyset$,
both ${j}_h$ and $J_h$ are proper, regular embeddings; thus the Gysin map ${j}^!_h$ and $J^!_h$ are
well-defined. 

\begin{lemm}\label{7.5}
The constant $c$ in in \eqref{c0} is given by the degree $c=\deg e({j}_h\sta \barV_1/{j}_h\sta \barS_\circ)$.
\end{lemm}

\begin{proof}
Let $\phi_h: \barW_h\to [h]$ be the projection to the point, and let $\phi_{(d)}$ be the projection
that fits into the Cartesian product
$$\begin{CD}
\barW_h @>{{j}_h}>> Z_{(d)}= \barW_{(d)}\cup \overline F_{(d)}\\
@VV{\phi_h}V @VV{\phi_{(d)}}V \\
[h] @>{\iota_h}>>   \ti\cB_{(d)}.
\end{CD}
$$
(Here $\phi_{(d)}$ is the composite of $Z_{(d)}\to \barW_{(d)}$ mentioned after \eqref{Zmu},
the morphism $\bar\gamma_{(d)}: \barW_{(d)}\to \tcX_{(d)}$, and the $\rho_{(d)}: \tcX_{(d)}\to \ti\cB_{(d)}$
constructed in \eqref{rhomu}.)

Since $j_h(\barW_h)\cap \overline F_{(d)} =\emptyset$, $j_h^![P_{(d),2}]=0$. Thus
$${j}_h^![P_{(d),1}]={j}_h^!([P_{(d),1}]+[P_{(d),2}])={j}_h^! 0_{\barV_{1,(d)}}^![R_{(d)}].
$$
Since Gysin maps commute, we obtain
$${j}_h^![P_{(d),1}]={j}_h^! 0_{\barV_{1,(d)}}^![R_{(d)}]
=0^!_{{j}_h\sta \barV_1}[J_h^! R_{(d)}]=0_{{j}_h\sta\barV_1}^![\barS_\circ|_{\barW_h}]
=e({j}_h\sta \barV_1/{j}_h\sta \barS_\circ).
$$
On the other hand, since the Gysin maps commute with proper push-forwards,
$$\phi_{h\ast}\bl e({j}_h\sta \barV_1/{j}_h\sta \barS_\circ)\br=\phi_{h\ast}\bl {j}_h^![P_{(d),1}]\br=
 \imath_h^!\bl\phi_{(d)\ast}  [P_{(d),1}]\br=\imath_h^!\bl c[\ti\cB_{(d)}]\br=c.
$$
This proves the Lemma.
\end{proof}

Finally, we evaluate $c$.
Let $\cE\to \barM_{1,1}$ with $s_1: \barM_{1,1}\to \cE$ be the universal family of 
(the moduli of pointed elliptic curves) $\barM_{1,1}$.
We form
$$\varphi: B=\barM_{1,1}\times\Po\mapright{\cong} M_h\sub M_\circ
$$
as follows. Let $q_1$ and $q_2$ be the first and the second projections of $B$.
We denote $q_1\sta\cE=\cE\times_{\barM_{1,1}} B$ with $q_1\sta s_1: B\to q_1\sta\cE$ be
the pullback section of $s_1$. Over $\Po$, we form $\cR=\Po\times\Po\to \Po$ (the second projection)
the constant family of $\Po$ over $\Po$, and let $s_2: \Po\to\cR$ be its section so that the image $s_2(\Po)\sub \Po\times\Po$ is
the diagonal. We let $q_2\sta\cR\to B$ with $q_2\sta s_2: B\to  q_2\sta\cR$ the pullback family with section.
We then glue $q_1\sta s_1(B)\sub q_1\sta\cE$ with $q_2\sta s_2(B)\sub q_2\sta\cR$ to form a
$B$-family of genus one nodal curves, denoted by $\pi_B: \cC_B\to B$.
We define
$f_B: \cC_B\lra \Pf$
be the composite
$$f_B: \cC_B=q_1\sta\cE\cup q_2\sta \cR\mapright{\text{ctr.}} q_2\sta \cR \mapright{\pr} \Po\mapright{h}\Pf,
$$
where the first arrow is contracting $q_1\sta\cE$, namely it maps $q_1\sta\cE$ onto $q_2\sta s_2(B)\sub q_2\sta \cR$
and is the identity on $q_2\sta\cR$; $\pr$ is the projection to the fiber of the constant family $q_2\sta \cR\to\Po$,
and $h$ is the map we picked.

It is clear that $(\pi_B, f_B): \cC_B\to B\times\Pf$ is a family of stable morphisms in $\cX=\barM_1(\Pf,d)$, thus
induces a morphism $B\to\cX$ that factors through $M_h\sub \cX$ and induces an isomorphism $B\cong M_h$.
In the following, we will not distinguish $B$ from $M_h$; in particular, $(\cC_B,f_B)$ is the tautological family
on $B\cong M_h\sub \cX$.

Let $\cH_B=\pi_{B\ast}\omega_{\cC_B/B}\cong (R^1\pi_{B\ast} \sO_{\cC_B})\dual$ be the Hodge bundle over
$B$. Using the family $(\cC_B,f_B)$, and because of (\ref{diag-3}), Lemma \ref{7.2} and Lemma \ref{7.3} , we have
$$(V_1/S_\circ)|_{M_h}=\sH_B\dual\otimes_{\sO_B} q_2\sta(h\sta T_{\Pf}/T_{\Po})=\sH_B\dual\otimes_{\sO_B} q_2\sta N_{R/\Pf},
$$
where $R=h(\Po)\sub\Pf$, and $N_{R/\Pf}$ is the normal bundle to $R$ in $\Pf$.
Also, for the line bundle $L$ on $M$ defined before \eqref{VV} and giving $W=\text{Total}(L)$, we have
$$L_B\defeq L|_{M_h}=\pi_{B\ast}(f_B\sta\sO(-5)\otimes \omega_{\cC_B/B})\cong
\sH_B\otimes_{\sO_B}(q_2\sta s_2)\sta f_B\sta\sO(-5).
$$
Thus, $\barW_h=\PP_B(L_B\oplus \sO_B)$, and for $j_h: \barW_h\to Z_{(d)}\supset \barW_{(d)}$, we have
$D_B:=\PP(0\oplus \sO_B)=j_h\upmo(D_\infty)$.
Following the construction of $\barV_1$ and $\barS_\circ$, we have
\beq\label{after4years} j_h\sta \barV_1/ j_h\sta \barS_\circ \cong
\bar\gamma_h^\ast({q}_2^\ast N_{R/\Pf} \otimes \sH_B\dual)(-D_B).
\eeq

Let $\xi\in A^1\Po$ and  $\zeta\in A^1\barM_{1,1}$ be defined via
$c_1(\sH)=\frac{1}{24} \zeta$ and $c(T_\Po)=1+2\xi$, where $\sH$ is the Hodge bundle on $\barM_{1,1}$.
We calculate
$c(h^\ast T_\Pf)=1+5d{\xi}$, and $c(N_{R/\Pf})=
1+(5d-2){\xi}$.
Let ${\bar\xi}=\bar\gamma_h^\ast{q}_2^\ast {\xi}$ and ${\bar\zeta}=\bar\gamma_h^\ast {q}_1^\ast {\zeta}\in A^1 \barW_h$
be pullbacks of $\xi$ and $\zeta$ via $\barW_h$ to $\Po$ and to $\barM_{1,1}$, respectively.
We calculate
\beq\label{cp}
c(\bar\gamma_h^\ast {q}_2^\ast  N_{R/\Pf})=1+(5d-2){\bar\xi}\and c(\bar\gamma_h^\ast{q}_1^\ast \sH\dual )=1-\frac{1}{24}{\bar\zeta}.
\eeq

Let $F\in A^2\barW_h$ be the Poincare dual of the fiber class of $\bar\gamma_h:\barW_h\to M_h$.
Using ${\bar\xi}{\bar\zeta}=F$ and ${\bar\xi}^2={\bar\zeta}^2=0$, \eqref{cp} gives
$$ c\bl\bar\gamma_h^\ast({q}_2^\ast N_{R/\Pf}\otimes {q}_1^\ast \sH\dual)\br=
1+\bl(5d-2){\bar\xi}-\frac{1}{8}\bar\zeta)-\frac{5d-2}{12}\cdot F.
$$
Hence the euler class
\beq\label{shot} e(j_h\sta \barV_1/ j_h\sta \barS_\circ)=[-D_B]^3+\bl(5d-2){\bar\xi}-\frac{1}{8}\bar\zeta)\cdot [D_B]^2-  (5d-2)/12\cdot F\cdot [D_B].
\eeq
(Here we view $[D_B]\in A^1\barW_h$ as the Poincare dual of the cycle $D_B$ in $\barW_h$.)

Let $\tau: D_B\to B=M_h$ be the projection (isomorphism). We compute each term.
First, direct calculations give
$c_1[N_{D_B/\barW_h}]=5d \tau^\ast{q}_2^\ast{\xi}-\frac{1}{24}\tau^\ast{q}_1^\ast{\zeta}$,
$a\cdot [D_B]^2 =\frac{-1}{24}$, and $b\cdot [D_B]^2=5d$. Thus the middle term in \eqref{shot}
is $\frac{1}{12}-\frac{5d}{6}$. The $c_1[N_{D_B/\barW_h}]$ just calculated implies
{ $[D_B]\cdot \tau^\ast{q}_2^\ast{\xi}
=\frac{-1}{24}$ and
$[D_B]\cdot \tau^\ast {q}_1^\ast{\zeta}
=5d$.}
Using $[D_B]^2=c_1[N_{D_B/\barW_h}]$,  we obtain $-[D_B]^3=-[D_B]^2\cdot [D_B]=\frac{5d}{24}+\frac{5d}{24}=\frac{5d}{12}.$
Finally, we calculate
$-\frac{(5d-2)[F]}{12}\cdot(-[D_B])=\frac{5d-2}{12}$.
Hence (\ref{shot}) is equal to $-\frac{1}{12}$. By Lemma \ref{7.5}, $c=-\frac{1}{12}$.

\begin{proof}[Proof of Proposition \ref{prop3}]
By (\ref{Rmu}), Lemma \ref{6.4} and (\ref{c?}) we conclude
$$\deg 0_{\ti\si,\loc}^![\bC_{(d)}]= 0^!_{\barV_{2,(d)}}[P_{(d),1}]=c\cdot \deg c_{5d+1}(\ti\sK_{(d)})[\ti\cB_{(d)}]=\frac{(-1)^{5d}}{12} N_0(d)_Q.
$$
This completes the algebro-geometric proof of the hyperplane property of the reduced genus one GW-invariants of
the quintics.
\end{proof}

\bibliographystyle{amsplain}

\end{document}